\documentclass[sn-mathphys]{sn-jnl}

\jyear{2022}%

\usepackage{graphicx}
\usepackage{amsmath}
\usepackage{epstopdf} 
\usepackage{amssymb,amsmath,color}
\usepackage[utf8]{inputenc}

\theoremstyle{thmstyleone}%
\newtheorem{theorem}{Theorem}[section]
\newtheorem{prop}[theorem]{Proposition}%

\theoremstyle{thmstyletwo}%
\newtheorem{exm}[theorem]{Example}
\newtheorem{rem}[theorem]{Remark}

\theoremstyle{thmstylethree}%
\newtheorem{lem}[theorem]{Lemma}

\numberwithin{equation}{section}

\newcommand{\Sm}{\textrm{Sm}_\Theta(I)}
\newcommand{\D}{\textrm{D}(I)}

\begin{document}

\title[A Lanczos-like method for linear ODEs]{A Lanczos-like method for non-autonomous linear ordinary differential equations}


\author[1]{\fnm{Pierre-Louis} \sur{Giscard}}\email{giscard@univ-littoral.fr}
\equalcont{These authors contributed equally to this work.}

\author*[2,3]{\fnm{Stefano} \sur{Pozza}}\email{pozza@karlin.mff.cuni.cz}
\equalcont{These authors contributed equally to this work.}

\affil[1]{\orgdiv{\orgname{Universit\'e du Littoral C\^{o}te d’Opale}, EA2597-LMPA-Laboratoire de Math\'ematiques Pures et Appliqu\'ees Joseph Liouville}, \orgaddress{\city{Calais}, \country{France}}}

\affil*[2]{\orgname{Charles University}, \orgdiv{Faculty of Mathematics and Physics}, \orgaddress{\city{Prague 8}, \country{Czech Republic}}}

\affil[3]{\orgdiv{Associated member of ISTI}, \orgname{CNR}, \orgaddress{\city{Pisa}, \country{Italy}}, and
member of INdAM-GNCS group, Italy}


\abstract{The time-ordered exponential is defined as the function that solves a system of coupled first-order linear differential equations with generally non-constant coefficients. In spite of being at the heart of much system dynamics, control theory, and model reduction problems, the time-ordered exponential function remains elusively difficult to evaluate. The $\ast$-Lanczos algorithm is a (symbolic) algorithm capable of evaluating it by producing a tridiagonalization of the original differential system. In this paper, we explain how the $\ast$-Lanczos algorithm is built from a generalization of Krylov subspaces, and we prove crucial properties, such as the \emph{matching moment property}. 
A strategy for its numerical implementation is also outlined and will be subject of future investigation.}




\maketitle

\section{Introduction}
Let $t'\geq t\in I\subseteq \mathbb{R}$ be variables--called times for convenience--in an interval $I$, and $\mathsf{A}(t')$ be an $N\times N$ time-dependent matrix. For a fixed $t$ (usually $t=0$), the time-ordered exponential of $\mathsf{A}(t')$ is defined as the unique solution $\mathsf{U}(t',t)$ of the non autonomous system of linear ordinary differential equations
\begin{equation}\label{FundamentalSystem}
\mathsf{A}(t') \mathsf{U}(t',t)=\frac{d}{dt'}\mathsf{U}(t',t), \quad \mathsf{U}(t,t)=\mathsf{Id}, \quad t' \geq t,
\end{equation}
with $\mathsf{Id}$ the identity matrix.
 Note that $t$ represents the time on which the initial condition is given, and that the (unusual) notation $\mathsf{U}(t',t)$ will be useful later.
If the matrix $\mathsf{A}$ commutes with itself at all times, i.e., $\mathsf{A}(\tau_1)\mathsf{A}(\tau_2)-\mathsf{A}(\tau_2)\mathsf{A}(\tau_1)=0$ for all $\tau_1,\tau_2 \in I$, then the time-ordered exponential is given by the matrix exponential $\mathsf{U}(t',t)=\exp\left(\int_t^{t'} \mathsf{A}(\tau)\, \textrm{d}\tau\right).$
However, when $\mathsf{A}$ does not commute with itself at all times, the time-ordered exponential has no known explicit form in terms of $\mathsf{A}$ and is rather denoted 
$$
\mathsf{U}(t',t)=\mathcal{T}\exp\left(\int_t^{t'} \mathsf{A}(\tau)\, \textrm{d}\tau\right),
$$ 
with $\mathcal{T}$ the time-ordering operator \cite{dyson1952}. This expression, introduced by Dyson in 1952, is more a notation than an explicit form as the action of the time-ordering operator is very difficult to evaluate. In particular, $\mathsf{U}(t',t)$ does not have a Cauchy integral representation, and it cannot be evaluated via ordinary diagonalization.
It is unlikely that a closed form expression for $\mathsf{U}(t',t)$ in terms of $\mathsf{A}$ exists at all since even when $\mathsf{A}$ is $2\times 2$, $\mathsf{U}$ can involve very complicated special functions \cite{Xie2010,Hortacsu2018}.

Evaluating time-ordered exponentials is a central question in the field of system dynamics, in particular in quantum physics where $\mathsf{A}$ is the Hamiltonian operator. Situations where this operator does not commute with itself are routinely encountered \cite{Blanes2009}, and the departure of the time-ordered exponential from a straightforward matrix exponential is responsible for many peculiar physical effects  \cite{Autler1955,Shirley1965,Lauder1986}. 
Further applications are found via
differential Lyapunov and Riccati matrix equations, which frequently appear in control theory, filter design, and model reduction problems \cite{Reid63,kwaSiv72,Corless2003,Blanes15,BenEtAll17}. Indeed, the solutions of such differential equations involve time-ordered exponentials \cite{Kuvcera73,Abou2003,hached2018,KirSim19}.\\[-.5em]

{In \cite{GisPoz21}, we introduced a tridiagonal form for the matrix $\mathsf{A}(t')$ from which it is possible to express a time-ordered exponential via path-sum continued fractions of finite depth. More precisely, the described procedure formulates each element of a time-ordered exponential in terms of a finite and treatable number of scalar integro-differential equations.  
Such a tridiagonal form is obtained by using the so-called \emph{$\ast$-Lanczos algorithm}. The algorithm also appeared in \cite{GisPozInv19} as the paper's motivation. Despite being at the core of results in \cite{GisPoz21} and being the motivation of \cite{GisPozInv19}, the $\ast$-Lanczos algorithm construction and the proofs of related main properties have yet not appeared in a scientific journal. The present paper aims to solve such a gap in the literature as it constructs the $\ast$-Lanczos algorithm from a (generalized) Krylov subspace perspective, proves the related \emph{Matching Moment Property}, introduces a bound for the approximation error, and adds further results on the breakdown issue that may affect the method.}

\subsection{Existing analytical approaches: Pitfalls and drawbacks}
In spite of the paramount importance of the time-ordered exponential, it is usually omitted from the literature on matrix functions.
Until 2015, only two families of analytical approaches existed (numerical methods will be discussed in Section~\ref{SecNumerical}). The first one to have been devised relies on Floquet theory and necessitates $\mathsf{A}(t')$ to be periodic (see, e.g., \cite{Blanes2009}). This method transforms Eq.~(\ref{FundamentalSystem}) into an \emph{infinite} system of coupled linear differential equations with \emph{constant} coefficients. This system is then solved perturbatively at very low order, as orders higher than 2 or 3 are typically too involved to be treated. The second method was developed in 1954 by Wilhelm Magnus \cite{Magnus1954}. It produces an infinite series of nested commutators of $\mathsf{A}$ with itself at different times, the ordinary matrix exponential of which provides the desired solution $\mathsf{U}(t',t)$. 
Magnus series are very much in use nowadays \cite{Blanes2009}, especially because they guarantee that the approximation to $\mathsf{U}(t',t)$ is unitary in quantum mechanical calculations \cite{Blanes2009}.
Nevertheless, the Magnus series for $\mathsf{U}(t',t)$ has a small (even if not so restrictively) convergence domain; see \cite{Casas07} and also \cite{Feldman1984,Maricq1987,IseAl2000,MoaNie07,Sanchez2011}.

In 2015, P.-L. G. et al. proposed a third method to obtain time-ordered exponentials using graph theory and necessitating only the entries $\mathsf{A}(t')_{k\ell}$ to be bounded functions of time \cite{Giscard2015}. The method formulates any desired entry or group of entries of $\mathsf{U}(t',t)$ as a branched continued fraction of \emph{finite} depth and breadth.
It has been succesfully used to solve challenging quantum dynamic problems, see e.g. \cite{Balasubramanian2020,BonGis2020}. This approach is unconditionally convergent and it provides exact expressions in terms of a finite number of integrals and Volterra equations. However, it suffers from a complexity drawback. 
Indeed, it requires one to find all the simple cycles and simple paths of a certain graph $G$. These are the walks on $G$ which are not self-intersecting. Unfortunately, the problem of enumerating such walks is $\#$P-complete \cite{Flum2004}, hindering the determination of exact solutions in large systems that must be treated using a further property of analytical path-sums called scale-invariance \cite{BonGis2020}. The present work with the results in \cite{GisPozInv19,GisPoz21} solves this issue by transforming the original matrix onto a structurally simpler one on which the path-sum solution takes the form of an ordinary, finite, continued fraction.

\subsection{The non-Hermitian Lanczos algorithm: Background}
Consider the simpler case in which $\mathsf{A}$ is not time-dependent. The solution of \eqref{FundamentalSystem} is given by the matrix function $\exp(\mathsf{A}(t'-t))$ which can be numerically approximated in several different ways (see, e.g., \cite{MolVLo78,MolVLo03,HigBook08}). One possible method is the (non-Hermitian) Lanczos algorithm. Computing the $(k,\ell)$ element of $\exp(\mathsf{A})$ is equivalent to computing the bilinear form $\boldsymbol{e}_k^H\exp(\mathsf{A}) \,\boldsymbol{e}_\ell$,
with $\boldsymbol{e}_k, \boldsymbol{e}_\ell$ vectors from the canonical Euclidean basis, and $\boldsymbol{e}_k^H$ the usual Hermitian transpose (here it coincides with the transpose since the vector is real). The non-Hermitian Lanczos algorithm (e.g., \cite{Gut92,Gut94b,GolMeuBook10,LieStrBook13}) gives, when no breakdown occurs, the matrices  
$$ \mathsf{V}_n = [\boldsymbol{v}_0,\dots,\boldsymbol{v}_{n-1}], \quad 
   \mathsf{W}_n = [\boldsymbol{w}_0,\dots,\boldsymbol{w}_{n-1}],$$
whose columns are biorthonormal bases respectively for the Krylov subspaces 
$$\text{span}\{ \boldsymbol{e}_\ell,  \mathsf{A}\,\boldsymbol{e}_\ell, \dots, \mathsf{A}^{n-1}\, \boldsymbol{e}_\ell \}, 
\quad
\text{span}\{ \boldsymbol{e}_k,  \mathsf{A}^H\boldsymbol{e}_k, \dots, (\mathsf{A}^H)^{n-1}\, \boldsymbol{e}_k \}.$$
Note that for $\mathsf{A}$ Hermitian and $k=\ell$ we can equivalently use the Hermitian Lanczos algorithm (getting $\mathsf{V}_n = \mathsf{W}_n$).
The so-called (complex) \emph{Jacobi matrix} $\mathsf{J}_n$ is the tridiagonal symmetric matrix with generally complex elements obtained by
$$ \mathsf{J}_n = \mathsf{W}_n^H \mathsf{A} \,\mathsf{V}_n.$$
As described in \cite{GolMeuBook10}, we can use the approximation
\begin{equation}\label{eq:class:lancz:approx}
     \boldsymbol{e}_k^H\exp(\mathsf{A}) \boldsymbol{e}_\ell \approx \boldsymbol{e}_1^H\exp(\mathsf{J}_n) \boldsymbol{e}_1,
\end{equation}
which relies on the so-called \emph{matching moment property}, i.e., 
\begin{equation}\label{eq:classical:mmp}
    \boldsymbol{e}_k^H (\mathsf{A})^j \boldsymbol{e}_\ell = \boldsymbol{e}_1^H (\mathsf{J}_n)^j \boldsymbol{e}_1, \quad j=0,1,\dots, 2n-1;
\end{equation}
see, e.g., \cite{GolMeuBook10,LieStrBook13} for the Hermitian case, and \cite{PozPraStr16,PozPraStr18} for the non-Hermitian one.
 The approximation \eqref{eq:class:lancz:approx} is a model reduction in two senses. First, the size of $\mathsf{A}$ is much larger than $n$ -- the size of $\mathsf{J}_n$. Second, the structure of the matrix $\mathsf{J}_n$ is much simpler since it is tridiagonal. 
 
 Given a matrix $\mathsf{A}$ with size $N$, the Lanczos algorithm can be used as a method for its tridiagonalization (see, e.g., \cite{Par92}). Assuming no breakdown, the $N$th iteration of the non-Hermitian Lanczos with input the matrix $\mathsf{A}$ and a couple of vectors $\boldsymbol{v}, \boldsymbol{w}$ produces the tridiagonal matrix $\mathsf{J}_N$, and the biorthogonal square matrices $\mathsf{V}_N, \mathsf{W}_N$ so that
 \begin{equation}\label{eq:full:tridiag}
    \mathsf{A}^j = \mathsf{V}_N \, (\mathsf{J}_N)^j\,\mathsf{W}_N^H, \quad j=0,1,\dots ,
 \end{equation}
 giving the exact expression
 \begin{equation}\label{eq:exptridiag}
    \exp(\mathsf{A}) = \mathsf{V}_N \exp(\mathsf{J}_N)\, \mathsf{W}_N^H.
 \end{equation}
Theorem~\ref{thm:full:tridiag} which we prove in this work extends this result to time-ordered exponentials. 

The Lanczos approximation \eqref{eq:class:lancz:approx} is connected with several further topics, such as (formal) orthogonal polynomials, Gauss quadrature, continued fractions, the moment problem, and many others. Information about these connections and references to the related rich and vast literature can be found, e.g., in the monographs \cite{DraBook83,GolMeuBook10,LieStrBook13} and the surveys \cite{PozPraStr16,PozPraStr18}.
\bigskip

Inspired by approximation \eqref{eq:class:lancz:approx}, the $\ast$-Lanczos algorithm produces a model reduction of a time-ordered exponential by providing a time-dependent tridiagonal matrix $\mathsf{T}_n$ satisfying properties analogous to the ones described above \cite{GisPoz21}. Differently from the classical case, the $*$-Lanczos algorithm works on vector distribution subspaces and it has to deal with a non-commutative product.

The time-dependent framework in which the proposed method works is much more complicated than the (time-independent) Krylov subspace approximation given by the (classical) Lanczos algorithm. In this paper, we will not deal with the behavior of the $*$-Lanczos algorithm considering approximations and finite-precision arithmetic problems. 

\subsection{Outline}
The work is organized as follows: In Section~\ref{AlgoSection}, we build the $\ast$-Lanczos algorithm. The algorithm relies on a non-commutative $\ast$-product between generalized functions of two-time variables, which we describe in Section~\ref{ProdMoment}. Then, in Section~\ref{BuildT}, we state the main result, Theorem~\ref{thm:mmp}, which underpins the Lanczos-like procedure. The Theorem establishes that the first $2n$ $\ast$-moments of a certain tridiagonal matrix $\mathsf{T}_n$ match the corresponding $\ast$-moments of the original matrix $\mathsf{A}$.
Theorem~\ref{thm:mmp} is proved with the tools developed in the subsequent Subsection~\ref{subsec:OP}. Section~\ref{sec:conv} is devoted to the convergence and breakdown properties of the algorithm, while examples of its use are given in Section~\ref{Examples}. In Section~\ref{SecNumerical} we outline a way to implement the Lanczos-like procedure numerically and we evaluate its computational cost. Section \ref{sec:conc} concludes the paper.

\section{The $\ast$-Lanczos Algorithm}\label{AlgoSection}
\subsection{The $\ast$-product and $\ast$-moments}\label{ProdMoment}
In this section, we recall the definition and the main properties of the product introduced in \cite[Section 1.2]{GisPozInv19}, complementing them with new results and examples.

Let $t$ and $t'$ be two real variables. We consider the class $\D$ of all distributions which are linear superpositions of Heaviside theta functions and Dirac delta derivatives with smooth coefficients over $I^2$. That is, a distribution $d$ is in $\D$ if and only if it can be written as
\begin{equation}\label{eq:genfun}
d(t',t)=\widetilde{d}(t',t)\Theta(t'-t) + \sum_{i=0}^N \widetilde{d}_i(t',t)\delta^{(i)}(t'-t),  
\end{equation}
where $N\in \mathbb{N}$ is finite, $\Theta(\cdot)$ stands for the Heaviside theta function (with the convention $\Theta(0)=1$) and $\delta^{(i)}(\cdot)$ is the $i$th derivative of the Dirac delta distribution $\delta= \delta^{(0)}$. Here and from now on, a tilde over a function (e.g., $\widetilde{d}(t',t)$) indicates that it is an ordinary function \emph{smooth} in both $t'\in I$ and $t\in I$. 
Note that we consider distributions as defined by Schwartz (\cite{schwartz1952,schwartz1978}). Hence a distribution $f \in \D$ should be interpreted as a linear functional applied to test functions.

We can endow the class $\D$ with a non-commutative algebraic structure upon defining a product between its elements. For $f_1, f_2
\in \D$ we define the convolution-like $\ast$ product between $f_1(t',t)$ and $f_2(t',t)$ as
\begin{equation}\label{eq:def:*}
  \big(f_2 * f_1\big)(t',t) := \int_{-\infty}^{\infty} f_2(t',\tau) f_1(\tau, t) \, \text{d}\tau,
\end{equation}
that has as identity element the Dirac delta distribution, $1_\ast:=\delta(t'-t)$.
When $f(t',t) = f(t'-t)$ has bounded supporting set, the $\ast$-product $f \ast g$ (and $g \ast f$) is equivalent to the convolution product for distributions defined by Schwartz (\cite[\S~11]{schwartz1952} and \cite[Chapter~VI]{schwartz1978}). Since $\delta^{(i)}(t'-t)$ has bounded supporting set, given $f \in \D$, the $\ast$-product $\delta^{(i)}(t'-t) \ast f$ and $f \ast \delta^{(i)}(t'-t)$ are well-defined and are both elements of $\D$; see \cite{GisPozInv19} for further details. Moreover, it holds 
\begin{align*}
 \delta^{(i)}(t'-t) \ast \delta^{(j)}(t'-t) &= \delta^{(j)}(t'-t) \ast \delta^{(i)}(t'-t) = \delta^{(i+j)}(t'-t); \\ 
   \Theta(t'-t)\ast \delta'(t'-t) &= \delta'(t'-t) \ast \Theta(t'-t) = \delta(t'-t).
\end{align*}

Consider the subclass $\Sm$ of $\D$ comprising those distributions 
of the form
\begin{equation}\label{PSmForm}
f(t',t)=\tilde{f}(t',t)\Theta(t'-t).
\end{equation}
For $f_1,\,f_2\in \Sm$, the $\ast$-product between $f_1,f_2$ simplifies to
\begin{align*}
  \big(f_2 * f_1\big)(t',t) &= \int_{-\infty}^{\infty} \tilde{f}_2(t',\tau) \tilde{f}_1(\tau, t)\Theta(t'-\tau)\Theta(\tau-t) \, \text{d}\tau,\\ &=\Theta(t'-t)\int_t^{t'} \tilde{f}_2(t',\tau) \tilde{f}_1(\tau, t) \, \text{d}\tau,
\end{align*}
 which makes calculations involving such functions easier to carry out and shows that $\Sm$ is closed under  $\ast$-multiplication.
 Together with the arguments above, this proves that $\D$ is closed under $\ast$-multiplication. Hence, for $f \in \D$, we can define its $k$th \emph{$\ast$-power} $f^{\ast k}$ as the $k$ $\ast$-products $f \ast f \ast \dots \ast f$, with the convention $f^{\ast 0} = \delta(t'-t)$.
 First examples of $\ast$-powers are
 \begin{align}\label{eq:thpow}
  \Theta^{\ast k}(t'-t) &= \frac{(t'-t)^{k-1}}{(k-1)!}\Theta(t'-t); \\
  \left(\delta^{(j)}(t'-t)\right)^{\ast k} &= \delta^{(kj)}(t'-t).
 \end{align}
 
  Note that, for members of $\Sm$, the $\ast$-product reduces exactly to the Volterra composition, a product between smooth functions of two-variables developed by Volterra and Pérès \cite{Volterra1928}. Volterra composition in the form above has seen little use since the 1950s because of perceived defects, such as the lack of identity element, which find remedies in the theory of distributions.

We illustrate the $\ast$-product behavior with the following example.
\begin{exm}
 Let $f(t',t) = \tilde{f}(t',t)\Theta(t'-t) = 2\sin(t')t \, \Theta(t'-t)$. Then
$$ \left(f\ast\Theta(t'-t) \right) (t',t) = \int_t^{t'} 2\sin(t')\tau \, \text{d}\tau = \sin(t')(t'^2-t^2)\Theta(t'-t), $$
is the integral of $\tilde{f}$ with respect to $t$,
while 
$$ \left(\Theta(t'-t) \ast f \right) (t',t) = \int_t^{t'} 2\sin(\tau)t \, \text{d}\tau = 2t\left(\cos(t)-\cos(t')\right)\Theta(t'-t), $$
is the integral of $\tilde{f}$ with respect to $t'$.
On the other hand, the $\ast$-products
\begin{align*}
  \left(f\ast\delta'(t'-t) \right) (t',t) 
  &= \int_{-\infty}^{+\infty} 2\sin(t') \tau \, \Theta(t'-\tau)\delta'(\tau-t) \, \text{d}\tau \\
  &= -2\sin(t')\Theta(t'-t) + 2\sin(t')t' \,\delta(t'-t); \\
  \left(\delta'(t'-t)\ast f \right) (t',t) 
  &= \int_{-\infty}^{+\infty} 2\sin(\tau)t \, \Theta(\tau-t)\delta'(t'-\tau) \, \text{d}\tau \\
  &= 2\cos(t')t \,\Theta(t'-t) + 2\sin(t)t \, \delta(t'-t);
\end{align*}
are derived by the formulas
\begin{align*}
    \left(f\ast\delta'(t'-t) \right) (t',t) &= -\left(\frac{\partial}{\partial t} \tilde{f}(t',t)\right)\Theta(t'-t) + \tilde{f}(t',t')\delta(t'-t); \\
  \left(\delta'(t'-t) \ast f \right) (t',t)  &= \left(\frac{\partial}{\partial t'} \tilde{f}(t',t)\right)\Theta(t'-t) + \tilde{f}(t,t)\delta(t'-t);
\end{align*}
see \cite{schwartz1978,GisPozInv19}. From here, one can verify that
$$ \left(f\ast\delta'(t'-t) \right) \ast \Theta(t'-t) = f, \quad \Theta(t'-t) \ast \left(\delta'(t'-t) \ast f \right) = f .$$
\end{exm}
We will not discuss any further the $\ast$-product by a Dirac delta derivative since it would bring us too far from the paper's goals. More details and examples can be found in \cite{GisPozInv19}.

The $\ast$-product extends directly to distributions of $\D$ whose smooth coefficients depend on less than two variables. Indeed, consider a generalized function $f_3(t',t)=\tilde{f}_3(t')\delta^{(i)}(t'-t)$ with $i\geq -1$ and $\delta^{(-1)}= \Theta$. Then
\begin{align*}
\big(f_3 \ast f_1\big)(t',t)&= \tilde{f}_3(t')\int_{-\infty}^{+\infty}  \delta^{(i)}(t'-\tau)f_1(\tau, t) \, \text{d}\tau,\\
\big(f_1\ast f_3\big)(t',t)&=\int_{-\infty}^{+\infty}  f_1(t',\tau)\tilde{f}_3(\tau)\delta^{(i)}(\tau-t) \, \text{d}\tau.
\end{align*}
where $f_1(t',t)$ is defined as before.
Hence the variable of $\tilde{f}_3(t')$ is treated as the left variable of a smooth function of two variables. This observation extends straightforwardly should $\tilde{f}_3$ be constant and, by linearity, to any distribution of $\D$.

The $\ast$-product also naturally extends to matrices whose entries are distributions of $\D$. Consider two of such matrices $\mathsf{A}_1(t',t)$ and $\mathsf{A}_2(t',t)\in \D^{N\times N}$ then
\begin{equation*}
  \big(\mathsf{A}_2 * \mathsf{A}_1\big)(t',t) := \int_{-\infty}^{+\infty} \mathsf{A}_2(t',\tau) \mathsf{A}_1(\tau, t) \, \text{d}\tau,
\end{equation*}
where the sizes of $\mathsf{A}_1, \mathsf{A}_2$ are compatible for the usual matrix product
(here and in the following, we omit the dependency on $t'$ and $t$ when it is clear from the context). As earlier, the $\ast$-product is associative and distributive with respect to the addition, but it is non-commutative. The identity element with respect to this product is now $\mathsf{Id}_\ast:=\mathsf{Id}\,1_\ast$, with $\mathsf{Id}$ the identity matrix of appropriate size.

Given a square matrix $\mathsf{A}(t',t)$ composed of elements from $\D$, we define the $k$-th \emph{matrix $\ast$-power} $\mathsf{A}^{\ast k}$ as the $k$ $\ast$-products $\mathsf{A}\ast \mathsf{A} \ast \dots \ast \mathsf{A}$. 
In particular, by \eqref{eq:thpow} we get the bound
\begin{equation*}
 \| \mathsf{A}^{*k}(t',t)\|_{\star} \leq \left(\sup_{\substack{\tau \geq \rho \\ \tau, \rho \in I}}\| \mathsf{A}(\tau,\rho)\|_{\star}\right)^{k} \frac{(t'-t)^{k-1}}{(k-1)!}\Theta(t'-t); \quad t',t \in I,
\end{equation*}
with $\| \cdot \|_\star$ any induced matrix norm. As a consequence,
the $\ast$-resolvent of any matrix depending on at most two variables is well defined, as  
$\mathsf{R}_{\ast}(\mathsf{A}):=\left(\mathsf{Id}_\ast-\mathsf{A}\right)^{\ast-1}=\mathsf{Id}_{\ast}+\sum_{k\geq 1} \mathsf{A}^{\ast k}$ exists provided  every entry of $\mathsf{A}$ is bounded for all $t',t \in I$ (see \cite{Giscard2015}). Then 
\begin{equation}\label{OrderedExp}
\mathsf{U}(t',t)=\Theta(t'-t) \ast \mathsf{R}_{\ast}(\mathsf{A})(t',t)
\end{equation}
is the time-ordered exponential of $\mathsf{A}(t',t)$; see \cite{Giscard2015}. Note that time-ordered exponentials are usually presented with only one-time variable, corresponding to $\mathsf{U}(t)=\mathsf{U}(t,0)$. Yet, in general $\mathsf{U}(t',t)\neq\mathsf{U}(t'-t,0)$.

In the spirit of the Lanczos algorithm, given a time-dependent matrix $\mathsf{A}(t',t)$, we will construct a matrix $\mathsf{T}_n(t',t)$ of size $n\leq N$ with a simpler (tridiagonal) structure and so that, fixing the indexes $k,\ell$, it holds
\begin{equation}\label{eq:prop:earlymmp}
    \big(\mathsf{A}^{*j}(t',t)\big)_{k,\ell} = \big(\mathsf{T}_n^{*j}(t',t)\big)_{1,1}, \quad \text{ for } \quad j=0,\dots, 2n-1, \quad t',t \in I;
\end{equation}
compare it with \eqref{eq:classical:mmp}.
In particular, when $n=N$ Property \eqref{eq:prop:earlymmp} stands for every $j\geq 0$, giving
$$ \mathsf{R}_{\ast}(\mathsf{A})_{k,\ell} = \mathsf{R}_{\ast}(\mathsf{T}_N)_{1,1}. $$
Hence the solution is given by the path-sum techniques exploiting the fact that the graph having $\mathsf{T}_N$ as its adjacency matrix is a path that admits self-loops.
More in general, given time-independent vectors $\boldsymbol{v},\boldsymbol{w}$ we call the \emph{$j$th $\ast$-moment} of $\mathsf{A},\boldsymbol{v},\boldsymbol{w}$ the scalar function $\boldsymbol{w}^H (\mathsf{A}^{*j}(t',t))\, \boldsymbol{v}$, for $j \geq 0$ (note that when the product is omitted, it stands for the usual matrix-vector product).
Then Property \eqref{eq:prop:earlymmp} is an instance of the more general case
\begin{equation*}
    \boldsymbol{w}^H (\mathsf{A}^{*j}(t',t))\, \boldsymbol{v} = \boldsymbol{e}_1^H (\mathsf{T}_n^{*j}(t',t))\, \boldsymbol{e}_1, \quad \text{ for } \quad j=0,\dots, 2n-1, \quad t',t \in I.
\end{equation*}

\subsection{Building up the $\ast$-Lanczos process}\label{BuildT}
Given a doubly time-dependent matrix  $\mathsf{A}(t',t)=\widetilde{\mathsf{A}}(t')\Theta(t'-t)$ and $k+1$ scalar generalized functions $\alpha_0(t',t), \alpha_1(t',t),\dots,\alpha_k({t',t}) \in \D$ which play the role of the \emph{coefficients}, we define the \emph{matrix $*$-polynomial} $p(\mathsf{A})(t',t)$ of degree $k$ as 
\begin{equation*}
    p(\mathsf{A})(t',t) := \sum_{j=0}^k \left(\mathsf{A}^{*j}*\alpha_j \right)(t',t);
\end{equation*}
moreover, we define the corresponding \emph{dual} matrix $\ast$-polynomial as
\begin{equation*}
    p^D(\mathsf{A})(t',t) := \sum_{j=0}^k \left(\bar{\alpha}_j*(\mathsf{A}^{*j}) \right)(t',t),
\end{equation*}
 where, in general, $\bar{d}$ is the conjugated value of $d \in \D$  and it is defined by conjugating the functions $\widetilde{d}$ and $\widetilde{d}_i$ in \eqref{eq:genfun}.
Let $\boldsymbol{v}$ be a time independent vector, we can define the set of time-dependent vectors $p(\mathsf{A})\boldsymbol{v}$, with $p$ a matrix $\ast$-polynomial. Such a set is a vector space with respect to the product $*$ and with scalars $\alpha_j(t',t)$ (the addition is the usual addition between vectors).
Similarly, given a vector $\boldsymbol{w}^H$ not depending on time, we can define the vector space given by the dual vectors $\boldsymbol{w}^H p^D(\mathsf{A})$. 
In particular, we can define the $\ast$-Krylov subspaces
\begin{align*}
    \mathcal{K}_n(\mathsf{A}, \boldsymbol{v})(t',t) &:= \left\{\, \left(\,p(\mathsf{A})\boldsymbol{v}\right)(t',t) \,\, \vert \,\, p \text{ of degree at most } n-1  \right\}, \\
    \mathcal{K}^D_n(\mathsf{A}, \boldsymbol{w})(t',t) &:= \left\{\, \left(\,\boldsymbol{w}^H p^D(\mathsf{A})\right)(t',t) \,\, \vert \,\, p \text{ of degree at most } n-1  \right\}. 
\end{align*}
The vectors  $\boldsymbol{v}, \mathsf{A}\boldsymbol{v}, \dots,\mathsf{A}^{*(n-1)}\boldsymbol{v}$ and $\boldsymbol{w}^H, \boldsymbol{w}^H \mathsf{A}, \dots,\boldsymbol{w}^H \mathsf{A}^{*(n-1)}$ are bases respectively for $\mathcal{K}_n(\mathsf{A}, \boldsymbol{v})$ and $\mathcal{K}^D_n(\mathsf{A}, \boldsymbol{w})$. 
We aim to derive \emph{$*$-biorthonormal}  bases $\boldsymbol{v}_0, \dots, \boldsymbol{v}_{n-1}$ and $\boldsymbol{w}_0^H, \dots, \boldsymbol{w}_{n-1}^H$ for the $\ast$-Krylov subspaces, i.e., so that
\begin{equation}\label{eq:v:orth:cond}
    \boldsymbol{w}_i^H * \boldsymbol{v}_j = \delta_{ij} \, 1_*,
\end{equation}
with $\delta_{ij}$ the Kronecker delta.

Assume that $\boldsymbol{w}^H \boldsymbol{v} = 1$, we can use a non-Hermitian Lanczos like biorthogonalization process for the triplet $(\boldsymbol{w}, \mathsf{A}(t',t), \boldsymbol{v})$. We shall call this method the \emph{$*$-Lanczos process}.
The first vectors of the biorthogonal bases are
\begin{equation*}
    \boldsymbol{v}_0 = \boldsymbol{v} \, 1_\ast, \quad
    \boldsymbol{w}_0^H = \boldsymbol{w}^H 1_\ast,
\end{equation*}
so that $\boldsymbol{w}_0^H * \boldsymbol{v}_0 = 1_*$.
Consider now a vector $\widehat{\boldsymbol{v}}_1 \in \mathcal{K}_2(\mathsf{A},\boldsymbol{v})$ given by
\begin{equation*}
    \widehat{\boldsymbol{v}}_1 = \mathsf{A}*\boldsymbol{v}_0 - \boldsymbol{v}_0 * \alpha_0=\mathsf{A}\boldsymbol{v} - \boldsymbol{v}\alpha_0.
\end{equation*}
If the vector satisfies the  $*$-biorthogonal condition $\boldsymbol{w}_0^H * \widehat{\boldsymbol{v}}_1 = 0$, then
\begin{equation}\label{eq:alpha0}
    \alpha_0 = \boldsymbol{w}_0^H * \mathsf{A} * \boldsymbol{v}_0=\boldsymbol{w}^H \mathsf{A}\,\boldsymbol{v}.
\end{equation}
Similarly, we get the expression
\begin{equation*}
    \widehat{\boldsymbol{w}}^H_1 = \boldsymbol{w}^H_0*\mathsf{A} - \alpha_0*\boldsymbol{w}^H_0=\boldsymbol{w}^H\mathsf{A} - \alpha_0\boldsymbol{w}^H,
\end{equation*}
with $\alpha_0$ given by \eqref{eq:alpha0}.
Hence the $*$-biorthonormal vectors are defined as
\begin{equation*}
    \boldsymbol{v}_1 = \widehat{\boldsymbol{v}}_1 * \beta_1^{*-1}, \quad 
    \boldsymbol{w}_1 = \widehat{\boldsymbol{w}}_1,
\end{equation*}
with $\beta_1=\widehat{\boldsymbol{w}}_1^H*\widehat{\boldsymbol{v}}_1$ and $\beta_1^{*-1}$ its \emph{$\ast$-inverse}, i.e., $\beta_1^{*-1}\ast\beta_1 = \beta_1 \ast \beta_1^{*-1}= 1_*$. We give sufficient conditions for the existence of such $\ast$-inverses below.
Assume now that we have the $*$-biorthonormal bases $\boldsymbol{v}_0, \dots, \boldsymbol{v}_{n-1}$ and $\boldsymbol{w}_0^H, \dots, \boldsymbol{w}_{n-1}^H$. Then we can build the vector
\begin{equation*}
    \widehat{\boldsymbol{v}}_n = \mathsf{A}*\boldsymbol{v}_{n-1} - \sum_{i=0}^{n-1} \boldsymbol{v}_{i}*\gamma_{i},
\end{equation*}
where the $\gamma_i$ are determined by the condition $\boldsymbol{w}_j^H * \widehat{\boldsymbol{v}}_n = \delta_{jn}1_\ast$, for $j=0,\dots,n-1$, giving
\begin{equation*}
    \gamma_j = \boldsymbol{w}_j^H * \mathsf{A} * \boldsymbol{v}_{n-1}, \quad j=0,\dots,n-1.
\end{equation*}
In particular, since $\boldsymbol{w}_j^H*\mathsf{A} \in \mathcal{K}_{j+1}^D(\mathsf{A},\boldsymbol{w})$ we get $\gamma_j = 0$ for $j=0,\dots,n-3$.
This leads to the following three-term recurrences for $n=1,2,\dots$ using the convention $\boldsymbol{v}_{-1} = \boldsymbol{w}_{-1} = 0$,
\begin{subequations}\label{Recurrences}
\begin{align}
    \boldsymbol{w}_n^H &= \boldsymbol{w}_{n-1}^H*\mathsf{A} - \alpha_{n-1}*\boldsymbol{w}_{n-1}^H - {\beta}_{n-1}*\boldsymbol{w}^H_{n-2},\\
    \boldsymbol{v}_n*\beta_{n} &= \mathsf{A}*\boldsymbol{v}_{n-1} - \boldsymbol{v}_{n-1}*\alpha_{n-1} - \boldsymbol{v}_{n-2},
\end{align}
\end{subequations}
with the coefficients given by
\begin{equation}\label{eq:lanc:coeff}
    \alpha_{n-1} = \boldsymbol{w}_{n-1}^H * \mathsf{A} * \boldsymbol{v}_{n-1}, \quad
    \beta_n = \boldsymbol{w}_n^H * \mathsf{A} * \boldsymbol{v}_{n-1}.
\end{equation}
Should $\beta_n$ not be $*$-invertible, we would get a \emph{breakdown} in the algorithm, since it would be impossible to compute $\boldsymbol{v}_n$. We developed a range of general methods to determine the $\ast$-inverse of functions of two-time variables which are gathered in \cite{GisPozInv19}.
These methods \emph{constructively} show the existence of $\beta_n^{\ast-1}$ almost everywhere on $I$ under the following conditions: 
\begin{itemize}
    \item $\beta_n\in \Sm$;
    \item $\beta_n\not \equiv 0$ on $I^2$.
\end{itemize}
Here the last condition means that $\beta_n$ is not identically null over $I^2$.
The question of whether or not all $\alpha_n, \beta_n\in \Sm$ was settled affirmatively in \cite{GisPoz21}. 
Let us define
$$\beta_n^{(1,0)}(t',t) := \frac{\partial}{\partial t} \beta_n(t',t), \quad t',t \in I. $$
In \cite{GisPoz21} we proved that if $\mathsf{A}(t',t) = \widetilde{\mathsf{A}}(t')\Theta(t'-t)$, composed of elements from $\Sm$, and if that $ \beta_n^{(1,0)}(t,t) \neq 0$ for every $t \in I$, 
then the coefficients $\alpha_n, \beta_n$ are in $\Sm$. The $\beta_n$ are thus $\ast$-invertible and, furthermore, their $\ast$-inverses take on a particular form (see also Theorem \ref{conj:coeff} in Subsection \ref{ConvSubsec}).
Morever given $\rho \in I$, $\beta_n^{(1,0)}(\rho,\rho)\neq0$ if and only if the (usual) non-Hermitian Lanczos algorithm does not breakdown when running on ${\mathsf{A}}(\rho), \boldsymbol{w}, \boldsymbol{v}$; see \cite{GisPoz21}.
Since the issue of breakdowns of the $\ast$-Lanczos algorithm is connected with the behavior of (usual) Lanczos techniques, we proceed as it is common when working with the non-Hermitian Lanczos algorithm. Thus, from now on, we assume all $\beta_n$ to be $*$-invertible and so that $\beta_n^{(1,0)}(t,t)\neq 0$ for every $t \in I$, while we come back to the issue of breakdowns in Section~\ref{subsec:breakdow}\\[-.5em]

The $*$-orthogonalization process described above defines the $*$-Lanczos algorithm (Table \ref{algo:*lan}).
The reason for this name is that the algorithm resembles the original Lanczos algorithm. Indeed, if all the inputs were \emph{time-independent}, and if we substituted $1_*$ with $1$ and the $*$ product with the usual matrix-vector or scalar-vector products, then Algorithm \ref{algo:*lan} would be mathematically equivalent to the non-Hermitian Lanczos algorithm.

\begin{table}
     \noindent\fbox{
\parbox{0.95\textwidth}{\begin{center}
\parbox{0.9\textwidth}{
\bigskip

 \noindent \underline{Input:} A complex time-dependent matrix $\mathsf{A}=\widetilde{\mathsf{A}}(t')\Theta(t'-t)$, and time-independent
                  complex vectors $\boldsymbol{v},\boldsymbol{w}$ such that $\boldsymbol{w}^H\boldsymbol{v} = 1$.

 \noindent \underline{Output:} Vectors $\boldsymbol{v}_0,\dots,\boldsymbol{v}_{n-1}$ and vectors $\boldsymbol{w}_0,\dots,\boldsymbol{w}_{n-1}$
  spanning respectively
  $\mathcal{K}_n(\mathsf{A},\boldsymbol{v})$,
  $\mathcal{K}_n(\mathsf{A},\boldsymbol{w})$
  and satisfying the $*$-biorthogonality conditions \eqref{eq:v:orth:cond}. The coefficients $\alpha_0, \dots, \alpha_{n-1}$ and $\beta_1,\dots,\beta_n$ from the recurrences \eqref{Recurrences}.  \begin{align*}
  & \textrm{Initialize: } \boldsymbol{v}_{-1}=\boldsymbol{w}_{-1}=0, \, \boldsymbol{v}_0 = \boldsymbol{v} \, 1_*, \, \boldsymbol{w}_0^H = \boldsymbol{w}^H 1_*. \\
  &    \alpha_0 = \boldsymbol{w}^H \mathsf{A} \, \boldsymbol{v}, \\
  &    \boldsymbol{w}_1^H = \boldsymbol{w}^H \mathsf{A} -  \alpha_{0}\, \boldsymbol{w}^H,\\
  &    \boldsymbol{\widehat v}_{1} = \mathsf{A}\, \boldsymbol{v} - \boldsymbol{v}\,\alpha_{0}, \\
  &    \beta_1 = \boldsymbol{w}^H \mathsf{A}^{*2} \, \boldsymbol{v} - \alpha_0^{*2}, \\
  &    \qquad\textrm{If } \beta_1 \textrm{ is not $*$-invertible, then stop, otherwise}, \\
  &     \boldsymbol{v}_1 = \boldsymbol{\widehat v}_1 * \beta_1^{*-1}, \\[1em]
  & \textrm{For } n=2,\dots \\ 
  &    \qquad \quad \alpha_{n-1} = \boldsymbol{w}_{n-1}^H * \mathsf{A} * \boldsymbol{v}_{n-1}, \\
  &    \qquad \quad \boldsymbol{w}_{n}^H = \boldsymbol{w}_{n-1}^H * \mathsf{A} -  \alpha_{n-1} * \boldsymbol{w}_{n-1}^H - \beta_{n-1}*\boldsymbol{w}_{n-2}^H, \\
  &    \qquad \quad \boldsymbol{\widehat v}_{n} = \mathsf{A} * \boldsymbol{v}_{n-1} - \boldsymbol{v}_{n-1}*\alpha_{n-1} - \boldsymbol{v}_{n-2}, \\
  &     \qquad \quad \beta_n = \boldsymbol{w}_{n}^H * \mathsf{A} * \boldsymbol{v}_{n-1}, \\
  &     \qquad \qquad\quad \textrm{If } \beta_n \textrm{ is not $*$-invertible, then stop, otherwise}, \\
  &     \qquad \quad  \boldsymbol{v}_n = \boldsymbol{\widehat v}_n * \beta_n^{*-1}, \\
  & \textrm{end}.
  \end{align*}
  }\end{center}}}
  \caption{$*$-Lanczos algorithm.}\label{algo:*lan}
  \end{table} 

Let us define the tridiagonal matrix
\begin{equation}\label{eq:tridiag}
    \mathsf{T}_n := \begin{bmatrix} 
            \alpha_0 & 1_* &          & \\ 
            \beta_1  & \alpha_1 & \ddots & \\ 
                     & \ddots   & \ddots & 1_* \\
                     &           & \beta_{n-1} & \alpha_{n-1}
          \end{bmatrix}, 
\end{equation}
and the matrices $\mathsf{V}_n:=[\boldsymbol{v}_0, \dots, \boldsymbol{v}_{n-1}]$ and $\mathsf{W}_n := [\boldsymbol{w}_0, \dots, \boldsymbol{w}_{n-1}]$. 
Then the three-term recurrences Eqs.~(\ref{Recurrences}) read, in matrix form,
\begin{align*}
    \mathsf{A} * \mathsf{V}_n = \mathsf{V}_n * \mathsf{T}_n + (\boldsymbol{v}_{n} * \beta_n) \boldsymbol{e}_n^H, \\
    \mathsf{W}_n^H * \mathsf{A} = \mathsf{T}_n *\mathsf{W}_n^H +  \boldsymbol{e}_n \, \boldsymbol{w}_n^H.
\end{align*}
Hence the tridiagonal matrix \eqref{eq:tridiag} can be expressed as
\begin{equation*}
    \mathsf{T}_n = \mathsf{W}_n^H * \mathsf{A} * \mathsf{V}_n.
\end{equation*}

The following property of $\mathsf{T}_n$ is fundamental for the time-ordered exponential approximation; its proof is given in Section \ref{subsec:OP}.
\begin{theorem}[Matching Moment Property]\label{thm:mmp}
   Let $\mathsf{A},\boldsymbol{w},\boldsymbol{v}$ and $\mathsf{T}_n$ be as described above, then
   \begin{equation}\label{AjTjResult}
        \boldsymbol{w}^H (\mathsf{A}^{*j})\, \boldsymbol{v} = \boldsymbol{e}_1^H (\mathsf{T}_n^{*j})\,\boldsymbol{e}_1, \quad \text{ for } \quad j=0,\dots, 2n-1.
   \end{equation}
\end{theorem}

Consider the time-ordered exponential $\mathsf{U}_n$ given by the differential equation 
\begin{equation}\label{eq:Un:def}
\mathsf{T}_n(t',t) \mathsf{U}_n(t',t)=\frac{d}{dt'}\mathsf{U}_n(t',t), \quad \mathsf{U}_n(t,t)=\mathsf{Id}.
\end{equation}
Theorem~\ref{thm:mmp} and Eq.~\eqref{OrderedExp} justify the use of the approximation 
\begin{equation}\label{eq:approx}
\boldsymbol{w}^H\mathsf{U}(t',t)\,\boldsymbol{v} \approx \boldsymbol{e}_1^H\mathsf{U}_n(t',t)\,\boldsymbol{e}_1
=\Theta(t'-t)\int_t^{t'} \mathsf{R}_{\ast}(\mathsf{T}_n)_{1,1}(\tau,t)\, \text{d}\tau;
\end{equation}
see \cite{GisPoz21}.
The system \eqref{eq:Un:def} can be seen as a reduced order model of the initial differential Eq. \eqref{FundamentalSystem} from two points of view.
First, $n$ may be much smaller than the size of $\mathsf{A}$; in this sense, in Section \ref{sec:conv}, we will discuss the convergence behavior of the approximation using Theorem \ref{thm:mmp}.
Secondly, looking at $\mathsf{A}$ and $\mathsf{T}_n$ as adjacency matrices, $\mathsf{A}$ may correspond to a graph with a complex structure, while $\mathsf{T}_n$ corresponds to a very simple graph composed of one path with possible self-loops.
Then the path-sum method gives
\begin{equation}\label{PSresult}
\mathsf{R}_{\ast}(\mathsf{T}_n)_{1,1}(t',t) 
= \Big(1_\ast - \alpha_0-\big(1_\ast-\alpha_1-(1_\ast-...)^{\ast-1}\ast\beta_2\big)^{\ast-1}\ast\beta_1\Big)^{\ast-1},
\end{equation}
see \cite{Giscard2012,Giscard2015}. This expression is analogous to the one for the first diagonal entry of the inverse of an ordinary tridiagonal matrix \cite{Kilic2008} (see also \cite{GolMeuBook10,LieStrBook13} for Jacobi matrices), except here all operations are taken with respect to the $\ast$-product.

For $n=N$, we get 
\begin{equation}\label{eq:orth:WV}
\mathsf{V}_N * \mathsf{W}_N^H = \mathsf{W}_N^H * \mathsf{V}_N = \mathsf{Id} \, 1_*.
\end{equation}
As a consequence 
\begin{equation}\label{eq:exact}
        \boldsymbol{w}^H (\mathsf{A}^{*j})\, \boldsymbol{v} = \boldsymbol{e}_1^H (\mathsf{T}_{N}^{*j})\,\boldsymbol{e}_1, \quad \text{ for } \quad j=0,1,\dots,
\end{equation}
and therefore the approximation \eqref{eq:approx} is actually exact. More generally, for $n=N$, the $\ast$-Lanczos algorithm produces the full tridiagonalization of $\mathsf{A}$ and of its time-ordered exponential, in a manner analogous to \eqref{eq:full:tridiag} and \eqref{eq:exptridiag}; see \cite{GisPoz21}.
\begin{theorem}\label{thm:full:tridiag}
   Let $\mathsf{A},\mathsf{V}_N,\mathsf{W}_N$ and $\mathsf{T}_N$ be as described above, then
   $$ \mathsf{A}^{\ast j} = \mathsf{V}_N \ast \mathsf{T}_N^{\ast j} \ast \mathsf{W}_N^H, \quad j = 0,1, \dots,  $$
   and thus
   $$ \mathsf{R}_\ast(\mathsf{A}) = \mathsf{V}_N \ast \mathsf{R}_\ast(\mathsf{T}_N) \ast \mathsf{W}_N^H.  $$
\end{theorem}
The theorem follows by using \eqref{eq:orth:WV}. Here, any entry of $R_\ast(T_N)$ is calculable using a path-sum continued of depth at most $N$.

\begin{rem}
The Lanczos-like method presented here for the time-ordered exponential is immediately valid for the ordinary matrix exponential function, since the latter is obtained from the former in the situation where $\mathsf{A}$ commutes with itself at all times,
$$
\mathcal{T}e^{\int \mathsf{A}(\tau) \, \text{d}\tau} = e^{\int_t^{t'} \mathsf{A}(\tau) \, \text{d}\tau}.
$$
This situation includes the case where $\mathsf{A}$ is time-independent, in which case setting $t=0$ and $t'=1$ above yields the matrix exponential of $\mathsf{A}$. However, the $\ast$-Lanczos algorithm cannot be considered a generalization of the Lanczos algorithm since its outputs on constant matrices are made of distributions and time dependent functions.
\end{rem}

\subsection{Matching $\ast$-moments through $*$-biorthonormal polynomials}\label{subsec:OP}
In order to prove Theorem \ref{thm:mmp}, 
we will exploit the connection between the $*$-Lanczos algorithm and families of $*$-biorthonormal polynomials. 
Let us define the set of $*$-polynomials 
\begin{equation*}
    \mathcal{P}_* := \left\{ p(\lambda) = \sum_{j=0}^k \lambda^{*j}*\gamma_j(t',t) \right\},
\end{equation*}
with $\gamma_j(t',t) \in \D$.
Consider a $*$-sesquilinear form $[\cdot, \cdot]: \mathcal{P}_* \times \mathcal{P}_* \rightarrow \D$, i.e., so that given $p_1,p_2,q_1,q_2 \in \mathcal{P}_\ast$ and $\alpha, \beta \in \D$, it satisfies 
\begin{align*}
    [q_1*\alpha, p_1*\beta] &= \bar{\alpha}*[q_1,p_1]*\beta, \\
    [q_1 + q_2, p_1 + p_2] &= [q_1, p_1] + [q_2, p_1] + [q_1, p_2] + [q_2, p_2].
\end{align*}
From now on we assume that every considered $*$-sesquilinear form $[\cdot, \cdot]$ also satisfies
\begin{equation}\label{eq:bfor:prop}
    [\lambda*q,p] = [q, \lambda*p].
\end{equation}
The $*$-sesquilinear form $[\cdot,\cdot]$ is determined by its \emph{$*$-moments} defined as
\begin{equation*}
    m_j(t,t'):= [\lambda^{*j},1] = [1, \lambda^{*j}], \quad j=0,1,\dots \,. 
\end{equation*}

We aim to build sequences of $\ast$-polynomials $p_0, p_1, \dots$ and $q_0, q_1, \dots$ so that they are $*$-biorthonormal with respect to $[\cdot,\cdot]$, i.e., 
\begin{equation}\label{eq:p:orth:cond}
    [q_i, p_j] = \delta_{ij} 1_*,
\end{equation}
where the subindex $j$ in $p_j$ and $q_j$ corresponds to the degree of the $\ast$-polynomial.
Here and in the following we assume $m_0 = 1_*$, getting $p_0 = q_0 = 1_*$.
Consider the $\ast$-polynomial
\begin{equation*}
   q_1(\lambda) = \lambda * q_0(\lambda) - q_0(\lambda)*\bar{\alpha}_0.
\end{equation*}
The orthogonality conditions \eqref{eq:p:orth:cond} give $\alpha_0 = [\lambda*q_0, p_0]$.
Similarly, we get the $\ast$-polynomial
\begin{equation*}
   p_1(\lambda) * \beta_1 = \lambda * p_0(\lambda) - p_0(\lambda) * \alpha_0,
\end{equation*}
with $\alpha_0 = [q_0, \lambda*p_0], \beta_1 = [q_1, \lambda*p_0]$.
Repeating the $*$-orthogonalization process, we obtain the three-term recurrences for $n=1,2,\dots$ 
\begin{subequations}\label{eq:p:rec}
\begin{align}
     q_n(\lambda) &= \lambda * q_{n-1}(\lambda) - q_{n-1}(\lambda)*\bar{\alpha}_{n-1} - q_{n-2}(\lambda)*\bar{\beta}_{n-1} \label{eq:p:rec:1} \\
     p_n(\lambda)*\beta_n &= \lambda * p_{n-1}(\lambda) - p_{n-1}(\lambda)*\alpha_{n-1} - p_{n-2}(\lambda),
     \label{eq:p:rec:2}
\end{align}
\end{subequations}
with $p_{-1} = q_{-1} = 0$ and
\begin{equation}\label{eq:p:coeff}
    \alpha_{n-1} = [q_{n-1}, \lambda * p_{n-1}],
    \quad
    \beta_{n} = [q_n, \lambda * p_{n-1}].
\end{equation}
Note that deriving the recurrences needs property \eqref{eq:bfor:prop}.
The $*$-biorthonormal polynomials $p_0, \dots, p_n$ and $q_0, \dots, q_n$ exist under the assumption that $\beta_1, \dots, \beta_n$ are $*$-invertible.

Let $\mathsf{A}$ be a time-dependent matrix, and $\boldsymbol{w},\boldsymbol{v}$ time-independent vectors such that $\boldsymbol{w}^H\boldsymbol{v} = 1$.
Consider the $*$-sesquilinear form $[\cdot, \cdot]$ defined by
\begin{equation*}
    [q, p] = \boldsymbol{w}^H \, q^D(\mathsf{A}) * p(\mathsf{A}) \, \boldsymbol{v}.
\end{equation*}
Assume that there exist $*$-polynomials $p_0, \dots, p_n$ and $q_0,\dots,q_n$ $*$-biorthonormal with respect to $[\cdot, \cdot]$.
Defining the vectors
\begin{equation*}
    \boldsymbol{v}_j = p_j(\mathsf{A})\,\boldsymbol{v}, \qquad \boldsymbol{w}^H_j = \boldsymbol{w}^H\, q_j^D(\mathsf{A}), 
\end{equation*}
and using the recurrences \eqref{eq:p:rec} gives the $*$-Lanczos recurrences \eqref{Recurrences}.
Moreover, the coefficients in \eqref{eq:p:coeff} are the $*$-Lanczos coefficients in \eqref{eq:lanc:coeff}.

 Let $\mathsf{T}_n$ be a tridiagonal matrix as in \eqref{eq:tridiag} composed of the coefficients \eqref{eq:p:coeff} associated with the $*$-sesquilinear form $[\cdot,\cdot]$. 
Then we can define the $*$-sesquilinear form
\begin{equation*}
    [q, p]_n = \boldsymbol{e}_1^H \, q^D(\mathsf{T}_n) * p(\mathsf{T}_n) \, \boldsymbol{e}_1.
\end{equation*}
The following lemmas show that 
\begin{equation*}
   m_j = [\lambda^{*j},1_*] = [\lambda^{*j},1_*]_n, \quad j=0,\dots,2n-1,
\end{equation*}
proving Theorem \ref{thm:mmp}.

\begin{lem}\label{lemma:op:basis}
   Let $p_0,\dots,p_n$ and $q_0,\dots,q_n$ be $*$-biorthonormal polynomials with respect to the $*$-sesquilinear form $[\cdot,\cdot]$.
   Assume that the coefficients $\beta_1,\dots, \beta_n$ in the related recurrences \eqref{eq:p:rec} are $*$-invertible.
   Then the $*$-polynomials are also $*$-biorthonormal with respect to the form $[\cdot,\cdot]_n$ defined above. 
\end{lem}
\begin{proof}
  Consider the vectors $\boldsymbol{y}_j^H = \boldsymbol{e}_1^H \, \mathsf{T}_n^{*j}$ and $\boldsymbol{x}_j= \mathsf{T}_n^{*j} \, \boldsymbol{e}_1$.
  Since the matrix $\mathsf{T}_n$ is tridiagonal, for $j=1,\dots,n-1$, we have 
  \begin{align*}
      \boldsymbol{e}_i^H \boldsymbol{x}_j = 0, \text{ for } i \geq j+2, &\quad \text{ and } \quad
      \boldsymbol{e}_{j+1}^H \boldsymbol{x}_j = \beta_j * \cdots * \beta_{1}, \\
      \boldsymbol{y}_j^H \boldsymbol{e}_i = 0, \text{ for } i \geq j+2, &\quad \text{ and } \quad
      \boldsymbol{y}_j^H \boldsymbol{e}_{j+1} = 1_* \,.
  \end{align*}
  By assumption, the product $\beta_j * \cdots * \beta_1$ is $*$-invertible. 
  Therefore there exist $*$-polynomials $\widehat{p}_0,\dots,\widehat{p}_{n-1}$ and $\widehat{q}_0,\dots,\widehat{q}_{n-1}$ so that,
  for $i=0,\dots,n-1$, we get
  \begin{equation*}
      1_* \boldsymbol{e}_{i+1}^H = \boldsymbol{e}_1^H\,\widehat{q}_i^D(\mathsf{T}_n), \qquad 
      1_* \boldsymbol{e}_{i+1} = \widehat{p}_i(\mathsf{T}_n)\, \boldsymbol{e}_1.
  \end{equation*}
Such $\ast$-polynomials are $*$-biorthonormal with respect to $[\cdot,\cdot]_n$ since they satisfy
\begin{equation*}
    [\widehat{q}_i, \widehat{p}_j]_n = 1_*\boldsymbol{e}_{i+1}^H \ast 1_* \boldsymbol{e}_{j+1} = \delta_{ij} 1_*.
\end{equation*}
Moreover, for $i=0,\dots,n-1$, the corresponding recurrence coefficients \eqref{eq:p:coeff} are the same as the ones of the $\ast$-polynomials $p_0,\dots,p_{n-1}$ and $q_0,\dots,q_{n-1}$. Indeed,
\begin{align*}
    \widehat{\alpha}_{i-1} &= [\widehat{q}_{i-1}, \lambda*\widehat{p}_{i-1}]_n = \boldsymbol{e}_{i-1}^H \mathsf{T}_n \, \boldsymbol{e}_{i-1} = \alpha_{i-1}, \\
    \widehat{\beta_{i}} &= [\widehat{q}_i, \lambda*\widehat{p}_{i-1}]_n = \boldsymbol{e}_{i}^H \, \mathsf{T}_n \boldsymbol{e}_{i-1} = \beta_{i}.
\end{align*}
Since $\widehat{p}_0 = p_0 = \widehat{q}_0 = q_0 = 1_*$, we get $\widehat{p}_i = p_i$ and $\widehat{q}_i = q_i$ for $i=0,\dots,n-1$. 
\end{proof}

\begin{lem}
   Let $p_0,\dots,p_{n-1}$ and $q_0,\dots,q_{n-1}$ be $*$-biorthonormal polynomials with respect to a $*$-sesquilinear form $[\cdot,\cdot]_A$
   and to a $*$-sesquilinear form $[\cdot,\cdot]_B$.
   If $[1_*,1_*]_A = [1_*,1_*]_B = 1_*$,
   then $[\lambda^{*j},1_*]_A = [\lambda^{*j},1_*]_B$ for $j=0,\dots,2n-1$.
\end{lem}
\begin{proof}
  We prove it by induction. 
  Let $m_j = [\lambda^{*j},1_*]_A$ and $\widehat{m}_j = [\lambda^{*j},1_*]_B$ for $j=0,1,\dots, 2n-1$.
  By the expression for the coefficients in \eqref{eq:p:coeff} we get
  \begin{equation*}
      [q_{0}, \lambda*p_{0}]_A = \alpha_0 = [q_{0}, \lambda*p_{0}]_B.
  \end{equation*}
  Hence $m_1 = \alpha_0 = \widehat{m}_1$.
  Assuming $m_j = \widehat{m}_j$ for $j=0,\dots,2k-3$
  we will prove that $m_{2k-2} = \widehat{m}_{2k-2}$
  and $m_{2k-1} = \widehat{m}_{2k-1}$, for $k = 2,\dots,n$.
  The coefficient expressions in \eqref{eq:p:coeff} gives
  \begin{equation*}
      [q_{k-1}, \lambda*p_{k-2}]_A = \beta_{k-1} = [q_{k-1}, \lambda*p_{k-2}]_B,
  \end{equation*}
  which can be rewritten as
  \begin{equation*}
      \sum_{i=0}^{k-1}\sum_{j=0}^{k-2} \bar{a}_i * m_{i+j+1} * b_j = \sum_{i=0}^{k-1}\sum_{j=0}^{k-2} \bar{a}_i * \widehat{m}_{i+j+1} * b_j,
  \end{equation*}
  with $a_i, b_j$ the coefficients respectively of $q_{k-1}$ and $p_{k-2}$.
  The inductive assumption implies
    \begin{equation*}
       \bar{a}_{k-1} * m_{2k-2} * b_{k-2} = \bar{a}_{k-1} * \widehat{m}_{2k-2} * b_{k-2}.
  \end{equation*}
  The leading coefficients of the $\ast$-polynomials $q_{2k-2}$ and $p_{2k-2}$ are respectively $a_{k-1}= 1_*$ and $b_{k-2}=(\beta_{k-2} * \cdots *\beta_{1})^{*-1}$. Hence $m_{2k-2} = \widehat{m}_{2k-2}$.
  Repeating the same argument with the coefficient $\alpha_{k-1}$ \eqref{eq:p:coeff} concludes the proof. 
\end{proof}

\section{Convergence, breakdown, and related properties}\label{sec:conv}
\subsection{The convergence behavior of intermediate approximations}\label{ConvSubsec}
Assuming no breakdown, the $\ast$-Lanczos algorithm in conjunction with the path-sum method converges to the solution $\boldsymbol{w}^H\mathsf{U}(t',t)\boldsymbol{v}$ in $N$ iterations, with $N$ the size of $\mathsf{A}$; see Eq.~\eqref{eq:exact}.
 Most importantly, intermediate $\ast$-Lanczos iterations provide a sequence of approximations $\int_t^{t'} \mathsf{R}_{\ast}(\mathsf{T}_n)_{1,1}(\tau,t)\, \text{d}\tau$, $n=1,\dots,N$, whose convergence behavior we analyze hereafter.
 
As we have already discussed before,
under the assumption that all entries of $\mathsf{A}$ are smooth over $I$ and that all the $\beta_j^{(1,0)}(t,t)\neq 0$, for every $t \in I$, all the $\alpha_j$ and $\beta_j$ distributions are elements of $\Sm$. The proof of this statement is very long and  technical and serves only to establish the theoretical \emph{feasibility} of tridigonalization for systems of coupled linear differential equations with variable coefficients using  smooth functions. It was therefore presented in a separate work. We refer the reader to \cite{GisPoz21} for a full exposition and proof, while here we only state some of the main results:
\begin{theorem}[P.-L. G. and S. P. \cite{GisPoz21}]\label{conj:coeff}
Let $\mathsf{A}(t',t)= \widetilde{\mathsf{A}}(t')\Theta(t'-t)$ be an $N\times N$ matrix composed of elements from $\Sm$. Let $\alpha_{n-1}$ and $\beta_n$ be the coefficients generated by Algorithm~\ref{algo:*lan} running on $\mathsf{A}$ and the time-independent vectors $\boldsymbol{w}, \boldsymbol{v}$ ($\boldsymbol{w}^H \boldsymbol{v}=1$). For any $1 \leq n \leq N$, assuming that $\beta_j^{(1,0)}(t,t) \neq 0$ for every $t \in I$,  $j=1,\dots,n-1$, we get $\beta_n(t',t), \alpha_{n-1}(t',t) \in \Sm$.
In addition, all required $\ast$-inverses $\beta_{j}^{\ast -1}$ exist and are of the form
$$
\beta_{j}^{\ast -1} = b_j^L(t',t) \ast \delta^{(3)}(t'-t),
$$
with $b_j^{L}\in \Sm$.
\end{theorem}

All $b_j^{L}$ have explicit expansions in terms of $\beta_j$ which are given in \cite{GisPoz21} but not reproduced here owing to length concerns. When $\beta_n \not\equiv 0$, but $\beta_n^{(1,0)}(\rho,\rho)=0$ for some $\rho \in I$, the results of Theorem \ref{conj:coeff} hold if we restrict the intial interval $I$ to a subinteval $J \subset I$ so that $\rho \notin J$. \\[-.5em]

Thanks to these regularity results, we can establish the following bound for the approximation error:
\begin{prop}\label{BoundConvergence}
Let us consider the setting and assumptions of Theorem \ref{conj:coeff}. Moreover, let $\mathsf{U}$ designate the time-ordered exponential of $\mathsf{A}$ and let $\mathsf{T}_n$ be the tridiagonal matrix \eqref{eq:tridiag} such that 
$$
\boldsymbol{w}^H\mathsf{A}^{*j}\boldsymbol{v}= (\mathsf{T}_n^{*j})_{1,1}, \quad \text{ for } \quad j=0,\dots, 2n-1.
$$
Then, for $t'\geq t$, with $t', t \in I$,
$$
\left\lvert\boldsymbol{w}^H\mathsf{U}(t',t)\boldsymbol{v}-\int_t^{t'} \mathsf{R}_{\ast}(\mathsf{T}_n)_{1,1}(\tau,t)\, \text{d}\tau\right\rvert\leq \frac{C^{2n}+D_n^{2n}}{(2n)!}(t'-t)^{2n}e^{(C+D_n)(t'-t)}.$$
Here 
$$C:=\sup_{t'\in I}\|{\mathsf{A}}(t')\|_{\infty}, \quad D_n:=3 \, \sup_{t',t\in I^2}\max_{0\leq j\leq n-1}\big\{\vert{\alpha}_j(t',t)\vert,\vert{\beta}_j(t',t)\vert\big\}$$
are both finite, with $\| \cdot \|_\infty$ the matrix norm induced by the uniform norm.. 
\end{prop}

\begin{proof}
Assume $t' \geq t$.
Observe that
\begin{align*}
\boldsymbol{w}^H\mathsf{U}(t',t)\boldsymbol{v}-\int_t^{t'} \!\!\mathsf{R}_{\ast}(\mathsf{T}_n)(\tau,t)_{11}\, \text{d}\tau=\int_t^{t'}\!\!\sum_{j=2n}^{\infty}\boldsymbol{w}^H\mathsf{A}^{\ast j}(\tau,t)\boldsymbol{v}-(\mathsf{T}_n^{\ast j})_{1,1}(\tau,t) \, \text{d}\tau,
\end{align*}
so that 
\begin{align*}
\left\lvert\boldsymbol{w}^H\mathsf{U}(t',t)\boldsymbol{v}-\!\int_t^{t'} \!\!\mathsf{R}_{\ast}(\mathsf{T}_n)(\tau,t)_{11}\,\text{d}\tau\right\rvert&\leq\int_t^{t'}\!\!\sum_{j=2n}^{\infty}\left\lvert\boldsymbol{w}^H\mathsf{A}^{\ast j}(\tau,t)\boldsymbol{v}\right\rvert+\left\lvert(\mathsf{T}_n^{\ast j})_{1,1}(\tau,t)\right\rvert\, \text{d}\tau.
\end{align*}
Now $\sup_{t'\in I}\vert\boldsymbol{w}^H{\mathsf{A}}(t')\boldsymbol{v}\vert\leq C$ 
and  
$$\left\lvert\int_t^{t'}\boldsymbol{w}^H\mathsf{A}^{\ast j}(\tau,t)\boldsymbol{v}\, \text{d}\tau\right\rvert\leq \Theta(t'-t) \ast \big(C\Theta(t'-t)\big)^{\ast j} = C^{j} \frac{(t'-t)^{j}}{j!}.$$ 
We proceed similarly for the terms involving $\mathsf{T}_n$.
Theorem \ref{conj:coeff} implies the existence of $\widehat{D}_n := \sup_{t',t\in I^2}\max_{0\leq j\leq n-1}\big\{\vert{\alpha}_j(t',t)\vert,\vert{\beta}_j(t',t)\vert\big\} < +\infty$. The matrix element $(\mathsf{T}_n^{\ast j})_{1,1}$ is given by the sum of $\ast$-products of coefficients $\alpha_i, \beta_i$ and $1_*$. Replacing all the factors in those $\ast$-products with $\widehat{D}_n \Theta(t'-t)$ gives an upper bound for $(\mathsf{T}_n^{\ast j})_{1,1}$. Hence we get
$$\left\lvert\left(\mathsf{T}_n^{\ast j}\right)_{1,1}\right\rvert \leq \left(\left(\widehat{D}_n \mathsf{P}_n\Theta(t'-t)\right)^{\ast j} \right)_{1,1} 
\leq \widehat{D}_n^j \|\mathsf{P}_n\|_{\infty}^j \Theta(t'-t)^{\ast j}, $$ 
where $\mathsf{P}_n$ is the $n\times n$ tridiagonal matrix whose nonzero entries are equal to $1$.
Note that $\|\mathsf{P}_n\|_{\infty}=3$.
Hence the error can be bounded by
\begin{align*}
\sum_{j=2n}^\infty \left(C^{j}+ D_n^j\right) \frac{(t'-t)^{j}}{j!} 
&\leq \frac{(C^{2n}+ D_n^{2n})}{2n!}(t'-t)^{2n} \sum_{j=0}^\infty \frac{2n!}{({2n+j})!} \left(C^{j}+ D_n^j\right) (t'-t)^{j}, \\
&\leq \frac{(C^{2n}+ D_n^{2n})}{2n!}(t'-t)^{2n} \sum_{j=0}^\infty  \frac{(C+ D_n)^j(t'-t)^{j}}{j!}, \\
&\leq \frac{(C^{2n}+ D_n^{2n})}{2n!}(t'-t)^{2n} e^{(C+D_n)(t'-t)},
\end{align*}
concluding the proof. 
\end{proof}

Analogously to the classical non-Hermitian Lanczos algorithm, we need to further assume $D_n$ to be not too large for $n\geq 1$ in order to get a meaningful bound. Such an assumption can be verified a-posteriori. 
The bound of Proposition~\ref{BoundConvergence} 
demonstrates that under reasonable assumptions, the approximation error has a super-linear decay. 
Assuming no breakdown, we also recall that the algorithm does necessarily converge in at most $N$ steps, independently of the error bound.
The  computational cost of the algorithm is discussed separately in Section~\ref{SecNumerical}.

\subsection{Breakdown}\label{subsec:breakdow}
In the classical non-Hermitian Lanczos algorithm, a breakdown appears either when an invariant Krylov subspace is produced (\emph{lucky breakdown}) or when the last vectors of the biorthogonal bases $\boldsymbol{v}_n, \boldsymbol{w}_n$ are nonzero, but $\boldsymbol{w}_n^H\boldsymbol{v}_n = 0$ (\emph{serious breakdown}); for further details refer, e.g., to \cite{Tay82,ParTayLiu85,Par92,Gut92,Freund1993,Gut94b}.
Analogously, in the $*$-Lanczos algorithm \ref{algo:*lan}, a \emph{lucky breakdown} arises when either $\boldsymbol{w}_n \equiv 0$ or $\widehat{\boldsymbol{v}}_n \equiv 0$.
In such a case, the algorithm has converged to the solution, as the following proposition shows.
\begin{prop}
  Assume that the $\ast$-Lanczos algorithm in Table \ref{algo:*lan} does not breakdown until the $n$th step when a lucky breakdown arises, i.e., $\widehat{\boldsymbol{v}}_n \equiv 0$ (or $\boldsymbol{w}_n \equiv 0$). Then
  \begin{align*}
        &\boldsymbol{w}^H (\mathsf{A}^{*j})\, \boldsymbol{v} = \boldsymbol{e}_1^H (\mathsf{T}_n^{*j})\,\boldsymbol{e}_1, \quad \text{ for } j\geq 0,\\
&\boldsymbol{w}^H\mathsf{U}(t',t)\boldsymbol{v} = \int_0^t\mathsf{R}_\ast(\mathsf{T}_n)_{1,1}(\tau,t)\,\text{d}\tau.
\end{align*}
\end{prop}
\begin{proof}
We prove it for $\widehat{\boldsymbol{v}}_n \equiv 0$. The case $\boldsymbol{w}_n \equiv 0$ follows similarly.
By the results in Subsection \ref{subsec:OP}, there exists a $\ast$-polynomial $\widehat{p}_n(\lambda) = \sum_{j=0}^n\lambda^{\ast j} \ast \gamma_j$, so that $\widehat{p}_n(\mathsf{A}) \boldsymbol{v} = \widehat{\boldsymbol{v}}_n \equiv 0$.
Therefore
\begin{equation*}
\mathsf{A}^{*n}\,\boldsymbol{v} = - \sum_{j=0}^{n-1}\mathsf{A}^{\ast j}\,\boldsymbol{v} \ast \gamma_j.
\end{equation*}
Hence, in general, for every $k\geq n$ there exist a $\ast$-polynomial $r_{n-1}$ of degree $n-1$ so that $\mathsf{A}^{\ast k} \, \boldsymbol{v} = r_{n-1}(\mathsf{A})\, \boldsymbol{v}$.

By Theorem \ref{thm:mmp} and using the notation of Subsection \ref{subsec:OP}, we get
$$ \boldsymbol{e}_1^H q_j^D(\mathsf{T}_n) \ast \widehat{p}_n(\mathsf{T}_n) \, \boldsymbol{e}_1 = \boldsymbol{w}^H q_j^D(\mathsf{A}) \ast \widehat{p}_n(\mathsf{A})\, \boldsymbol{v} = 0, \quad j=0,\dots,n-1. $$
As shown in the proof of Lemma \ref{lemma:op:basis}, $\boldsymbol{e}_1^H q_j^D(\mathsf{T}_n) = 1_* \boldsymbol{e}_{j+1}^H$, for $j=0,\dots,n-1$. Therefore we get $\widehat{p}_n(\mathsf{T}_n)\, \boldsymbol{e}_1 = 0$. As a consequence, for every $k \geq n$ we get
\begin{equation*}
\left(\mathsf{T}_n\right)^{*n}\,\boldsymbol{e}_1 = - \sum_{j=0}^{n-1}\left(\mathsf{T}_n\right)^{\ast j}\,\boldsymbol{e}_1 \ast \gamma_j,     
\end{equation*}
and thus $\left(\mathsf{T}_n\right)^{*k}\,\boldsymbol{e}_1 = r_{n-1}\left(\mathsf{T}_n\right) \, \boldsymbol{e}_1$.
Since $r_{n-1}$ has degree $n-1$, Theorem \ref{thm:mmp} implies
$$ \boldsymbol{w}^H \mathsf{A}^{*k}\,\boldsymbol{v} = \boldsymbol{w}^H r_{n-1}(\mathsf{A})\, \boldsymbol{v} = \boldsymbol{e}_1^H r_{n-1}(\mathsf{T}_n)\, \boldsymbol{e}_1 = \boldsymbol{e}_1^H \left(\mathsf{T}_n\right)^{*k}\,\boldsymbol{e}_1, $$
concluding the proof. 
\end{proof}



The $\ast$-Lanczos algorithm construction and its polynomial interpretation in Subsection \ref{subsec:OP} suggest that it may be possible to deal with the serious breakdown issue by a look-ahead strategy analogous to the one for the non-Hermitian Lanczos algorithm; see, e.g., \cite{Tay82,ParTayLiu85,BreRedSad91,Brezinski1992,Par92,Gut92,Freund1993,Gut94b,PozPra19}. 
Nevertheless, we need to discuss the particular case of serious breakdowns arising when $\boldsymbol{w}= \boldsymbol{e}_i$ and $\boldsymbol{v}=\boldsymbol{e}_j$. If $i\neq j$, then $\boldsymbol{w}^H\boldsymbol{v} = 0$, which does not satisfy the $\ast$-Lanczos assumption. Moreover, if $i=j$ and $\mathsf{A}$ is a sparse non-Hermitian matrix, then it may be possible that $\mathsf{A}_{ii} \equiv0$ and $\mathsf{A^{*2}}_{ii} \equiv 0$.
As a consequence, we get $\beta_1 \equiv 0$. We can try to fix these problems rewriting the approximation of the time-ordered exponential $\mathsf{U}$ as
$$\boldsymbol{e}_i^H\mathsf{U}\boldsymbol{e}_j = (\boldsymbol{e} + \boldsymbol{e}_i)^H \mathsf{U} \boldsymbol{e}_j - \boldsymbol{e}^H \mathsf{U} \boldsymbol{e}_j,
$$
with $\boldsymbol{e} = (1, \dots, 1)^H$. Then one can approximate $(\boldsymbol{e} + \boldsymbol{e}_i)^H \mathsf{U} \, \boldsymbol{e}_j$ and $\boldsymbol{e}^H \mathsf{U} \, \boldsymbol{e}_j$ separately, which are less likely going to have a breakdown, thanks to the fact that $\boldsymbol{e}$ is a full vector; see, e.g., \cite[Section 7.3]{GolMeuBook10}.

\section{Examples}\label{Examples}
In this section, we use the $\ast$-Lanczos algorithm~\ref{algo:*lan} on examples in ascending order of difficulty. All the computations have been performed using \textsc{Mathematica} 11.

\begin{exm}[Ordinary matrix exponential]
Let us first consider a constant matrix 
\begin{equation*}
\mathsf{A}=
\begin{pmatrix}
-1 & 1 & 1 \\
 1 & 0 & 1 \\
 1 & 1 & -1 
 \end{pmatrix}.
\end{equation*}
Because $\mathsf{A}$ commutes with itself at all times, its time-ordered exponential coincides with its ordinary exponential, $\mathcal{T}e^{\int \mathsf{A}(\tau)\, \text{d}\tau}\equiv e^{\mathsf{A}(t')}$ (we set $t=0$). Note that the matrix chosen here is symmetric only to lead to concise expressions suitable for presentation in an article, e.g.,
\begin{equation}\label{A11}
\big(e^{\mathsf{A}t'}\big)_{11}=-\frac{1}{2} \sinh (2 t')+\frac{1}{2} \cosh (2 t')+\frac{1}{2} \cosh \left(\sqrt{2}
   t'\right),
\end{equation} 
and that such symmetries are not a requirement of the $\ast$-Lanczos approach. 

Now let us find the result of Eq.~(\ref{A11}) with  Algorithm~\ref{algo:*lan}. We define $\boldsymbol{w}:=\boldsymbol{v}^H:=(1,0,0)$, $\boldsymbol{w}_0=\boldsymbol{w}1_\ast$, $\boldsymbol{v}_0=\boldsymbol{v}1_\ast$, from which it follows that $\alpha_0(t',t) = -1\times\Theta(t'-t)$ and $\boldsymbol{w}_1=\widehat{\boldsymbol{v}}_1^H=(0,1,1)\Theta(t'-t)$. Furthermore, since $\mathsf{A}$ is a constant matrix times $\Theta(t'-t)$, we have $\mathsf{A}^{\ast n} = \tilde{\mathsf{A}}^n \times \Theta(t'-t)^{\ast n} = \tilde{\mathsf{A}}\times(t'-t)^{n-1}/(n-1)!\times \Theta(t'-t)$ and similarly  $\alpha_0^{\ast2}(t',t)=\tilde{\alpha}_0^2\times (t'-t)\Theta(t'-t)$. Thus 
\begin{align*}
&\beta_1=\boldsymbol{w}^H\mathsf{A}^2\boldsymbol{v}\times (t'-t)\Theta(t'-t)-\tilde{\alpha}_0^2(t'-t)\Theta(t'-t) = 2(t'-t)\Theta(t'-t).
\end{align*}
The $\ast$-inverse follows as $\beta^{\ast-1}=\frac{1}{2} \delta''$ \cite{GisPozInv19},
from which we get 
$$
\boldsymbol{v}_1 = \widehat{\boldsymbol{v}}_1\ast \beta_1^{\ast-1} = (0,1,1)^H\frac{1}{2}\delta'(t'-t),
$$
Now it follows that 
\begin{align*}
&\alpha_1(t',t)=\boldsymbol{w}_1\ast \mathsf{A}\ast \boldsymbol{v}_1=\frac{1}{2}\Theta(t-t'),\\
&\boldsymbol{w}_2(t',t)=\boldsymbol{w}_1\ast \mathsf{A}-\alpha_1\ast \boldsymbol{w}_1-\beta_1\ast \boldsymbol{w}_0=(0,1,-1)\frac{1}{2}(t'-t)\Theta(t'-t),\\
&\widehat{\boldsymbol{v}}_2(t',t)=\mathsf{A}\ast\boldsymbol{v}_1- \boldsymbol{v}_1\ast\alpha_1-\boldsymbol{v}_0=(0,1,-1)^H\frac{1}{4}\delta(t'-t),\\
&\beta_2=\boldsymbol{w}_2\ast\mathsf{A}\ast\boldsymbol{v}_1 = \frac{1}{4}(t'-t)\Theta(t'-t).
\end{align*}
Then $\beta_2^{\ast-1}=4 \delta''$ and so
\begin{align*}
&\boldsymbol{v}_2 = \widehat{\boldsymbol{v}}_2\ast\beta_2^{\ast-1}=(0,1,-1)^H\delta''(t'-t)
&\alpha_2=\boldsymbol{w}_2\ast \mathsf{A}\ast\boldsymbol{v}_2=-\frac{3}{2}\Theta(t'-t).
\end{align*}
At this point we have determined the $\ast$-Lanczos matrices $\mathsf{T}$, $\mathsf{V}$ and $\mathsf{W}$ entirely
\begin{equation*}
\mathsf{T}=\begin{pmatrix}
 -\Theta & \delta & 0 \\
 2 \Theta^{\ast2} & \frac{1}{2}\Theta & \delta \\
 0 & \frac{1}{4}\Theta^{\ast2} & -\frac{3}{2}\Theta \\
\end{pmatrix}, \quad
\mathsf{V}=\begin{pmatrix}
 \delta & 0 & 0 \\
 0 & \frac{1}{2} \delta ' & \delta '' \\
 0 & \frac{1}{2} \delta ' & -\delta '' 
\end{pmatrix}, \quad 
\mathsf{W}^{H}=\begin{pmatrix}
 \delta & 0 & 0 \\
 0 & \Theta & \Theta \\
 0 & \frac{1}{2}\Theta^{\ast2} & -\frac{1}{2}\Theta^{\ast2}
\end{pmatrix}.
\end{equation*}
In all of these expressions, $\Theta$ is a short-hand notation for $\Theta(t'-t)$ and $\delta$, $\delta'$ and $\delta''$ are to be  evaluated in $t'-t$.
It is now straightforward to verify the matching moment property $\big(\mathsf{T}^{\ast j}\big)_{11}=\big(\mathsf{A}^{\ast j}\big)_{11}$ for all $j\in\mathbb{N}$. We can also check directly that the time-ordered exponential of $\mathsf{A}$ is correctly determined from $\mathsf{T}$ using either the general formula of Eq.~(\ref{PSresult}) or, because the situation is so simple that all entries depend only on $t'-t$, we may use a Laplace transform with respect to $t'-t$. This gives $\mathsf{T}(s)$, and the inverse Laplace-transform of the resolvent $\big(\mathsf{I}-\mathsf{T}(s)\big)^{-1}_{11}$ is the desired quantity. Both procedures give the same result, namely the derivative of $e^{\mathsf{A}t}$ as it should \cite{Giscard2015}, i.e.,
$$
(\mathsf{Id}_\ast-\mathsf{T})^{\ast-1}_{11}(t',0) = \left(\sinh (2 t')+\frac{1}{\sqrt{2}}\sinh \left(\sqrt{2} t'\right)-\cosh (2 t')\right)\Theta(t'),
$$
which is indeed the derivative of Eq.~(\ref{A11}).
\end{exm}

\begin{exm}[Time-ordered exponential of a time-dependent matrix]\label{Ex2}
In this example, we consider the  $5\times 5$ time-dependent matrix $\mathsf{A}(t',t)=\tilde{\mathsf{A}}(t')\Theta(t'-t)$ with
\begin{align*}
\tilde{\mathsf{A}}(t') =
\left(
\begin{array}{ccccc}
 \cos (t') & 0 & 1 & 2 & 1 \\
 0 & \cos (t')-t' & 1-3 t' & t' & 0 \\
 0 & t' & 2 t'+\cos (t') & 0 & 0 \\
 0 & 1 & 2 t'+1 & t'+\cos (t') & t' \\
 t' & -t'-1 & -6 t'-1 & 1-2 t' & \cos (t')-2 t' \\
\end{array}
\right).
\end{align*}
The matrix $\tilde{\mathsf{A}}$ does not commute with itself at different times $\tilde{\mathsf{A}}(t')\tilde{\mathsf{A}}(t)-\tilde{\mathsf{A}}(t)\tilde{\mathsf{A}}(t')\neq 0$, and the corresponding differential system Eq.~(\ref{FundamentalSystem}) has no known analytical solution.
We use Algorithm~\ref{algo:*lan} to determine the tridiagonal matching moment matrix $\mathsf{T}$ such that $\big(\mathsf{A}^{\ast j}\big)_{11}=\big(\mathsf{T}^{\ast j}\big)_{11}$ for $j \in \mathbb{N}$. We define $\boldsymbol{w}:=\boldsymbol{v}^H:=(1,0,0,0,0)$, $\boldsymbol{w}_0=\boldsymbol{w}1_\ast$, $\boldsymbol{v}_0=\boldsymbol{v}1_\ast$, from which it follows that  
\begin{align*}
\alpha_0(t',t) &= \cos(t')\Theta(t'-t),\\
\boldsymbol{w}_1&=\big(0,0,1,2,1\big)\Theta(t'-t),\\
\widehat{\boldsymbol{v}}_1&=\big(0,0,0,0,t')^H\Theta(t'-t),\\
\beta_1(t',t)&=\frac{1}{2} \left(t'^2-t^2\right)\Theta(t'-t).
\end{align*} 
Observing that $\beta_1 = \Theta(t'-t)\ast t'\Theta(t'-t)$, we get $\beta_1^{*-1} = \frac{1}{t} \delta'(t'-t)\ast \delta'(t'-t)=-\frac{1}{t^2}\delta'(t'-t)+\frac{1}{t}\delta''(t'-t)$, so that
$$
\boldsymbol{v}_1=\widehat{\boldsymbol{v}}_1\ast\beta_1^{\ast-1}=\big(0,\,0,\,0,\,0,\,1\big)^H\delta'(t'-t),
$$
which terminates the initialization phase of the Algorithm. We proceed with
\begin{align*}
\alpha_1(t',t)&=\boldsymbol{w}_1\ast\mathsf{A}\ast\boldsymbol{v}_1 = \cos(t)\Theta(t'-t),\\
   \boldsymbol{w}_2&=\boldsymbol{w}_1\ast\mathsf{A}-\alpha_1\ast \boldsymbol{w}_1-\beta_1\ast \boldsymbol{w}_0,\\
   &=\big(0,\,t'-t,\,t'-t,\,t'-t,\, 0\big)\Theta(t'-t),\\
   \widehat{\boldsymbol{v}}_2&=\mathsf{A}\ast\boldsymbol{v}_1-\boldsymbol{v}_1\ast\alpha_1 -\boldsymbol{v}_0= \big(0,\,0,\,0,\,t,\,-2t\big)^H\delta(t'-t),\\
   \beta_2&=\boldsymbol{w}_2\ast \mathsf{A}\ast \boldsymbol{v}_1 =t (t'-t)\Theta(t'-t).
\end{align*}
As we did for $\beta_1$, we factorize $\beta_2 = \Theta(t'-t) \ast t\, \Theta(t'-t)$ so that its $\ast$-inverse is $\beta_2^{*-1} = \frac{1}{t'} \delta'(t'-t)\ast \delta'(t'-t)=\frac{1}{t'}\delta''$.  Then
$$
\boldsymbol{v}_2 = \big(0,\,0,\,0,\,1,\,-2\big)^H\delta''(t'-t).
$$
Continuing in this fashion yields the tridiagonal output matrix $\mathsf{T}_5\equiv \mathsf{T}$,
$$\mathsf{T}\!\!=\!\!\begin{pmatrix}
 \cos (t') \Theta & \delta & 0 & 0 & 0 \\
 \frac{1}{2}(t'^2\!-\!t^2) \Theta & \cos (t) \Theta & \delta & 0 & 0 \\
 0 & t (t'\!-\!t) \Theta & \widetilde{\alpha}_2(t',t) \Theta & \delta & 0 \\
 0 & 0 & -\frac{1}{2}(3 t^2\!-\!4 t t'\!+\!t'^2) \Theta & \widetilde{\alpha}_3(t',t) \Theta & \delta \\
 0 & 0 & 0 & (-2 t^2 \!+\! 3 t t' \!-\! t'^2) \Theta & \widetilde{\alpha}_4(t',t) \Theta  \\
\end{pmatrix}\!\!,
$$
with
\begin{align*}
\widetilde{\alpha}_2(t',t) & = (t'-t) \sin (t)+\cos (t), \\
\widetilde{\alpha}_3(t',t) & =\frac{1}{2} \left(4 (t'-t) \sin(t)-\left((t-t')^2-2\right)\right)\cos(t), \\
\widetilde{\alpha}_4(t',t) &=\frac{1}{6} \left(\left((t-t')^2-18\right) (t-t') \sin (t)+\left(6-9 (t-t')^2\right) \cos (t)\right),
\end{align*}
and the bases matrices
$$
\mathsf{V}_5=
\left(
\begin{array}{ccccc}
 \delta & 0 & 0 & 0 & 0 \\
 0 & 0 & 0 & \delta ^{(3)} & -2 \delta ^{(4)} \\
 0 & 0 & 0 & 0 & \delta ^{(4)} \\
 0 & 0 & \delta ''& -\delta ^{(3)} & \delta ^{(4)} \\
 0 & \delta ' & -2 \delta ''& 2 \delta ^{(3)} & -3 \delta ^{(4)} \\
\end{array}
\right), \quad 
\mathsf{W}_5^H =
\left(
\begin{array}{ccccc}
 \delta & 0 & 0 & 0 & 0 \\
 0 & 0 & \Theta & 2 \Theta & \Theta \\
 0 & \Theta^{\ast2} & \Theta^{\ast2} & \Theta^{\ast2} & 0 \\
 0 & \Theta^{\ast3} & 2\Theta^{\ast3} & 0 & 0 \\
 0 & 0 & \Theta^{\ast4} & 0 & 0 \\
\end{array}
\right).
$$
In all of these expressions, $\Theta$ and $\delta^{(n)}$ are short-hand notations respectively for $\Theta(t'-t)$ and $\delta^{(n)}(t'-t)$. 
All the required $\beta_j^{*-1}$ were calculated using the strategies described in \cite{GisPozInv19}, getting the factorized $*$-inverses
\begin{align*}
\beta_3^{*-1} = \frac{1}{t}\Theta(t'-t) \ast \delta^{(3)}(t'-t), \quad
\beta_4^{*-1} = \frac{t'}{t^2} \Theta(t'-t) \ast \delta^{(3)}(t'-t).
\end{align*}
We have also verified that $\big(\mathsf{A}^{\ast j}\big)_{11}=\big(\mathsf{T}^{\ast j}\big)_{11}$ holds for $j$ up to 9. The $\ast$-resolvent of $\mathsf{T}$ has no closed-form expression, its Neumann series likely converging to a hitherto undefined special function. Ultimately, such difficulties are connected with the propensity of  systems of coupled linear ordinary differential equations with non-constant coefficients  to produce transcendent solutions.
\end{exm}

\section{Outlook: Numerical implementation}\label{SecNumerical}
We do not expect closed-forms to exist in most cases for the entries of time-ordered matrix exponentials as these can involve complicated special functions \cite{Xie2010}. Also, very large matrices $\mathsf{A}(t')$ are to be treatable by the algorithm for it to be relevant to most applications. 
For these reasons, it is fundamental to implement the $\ast$-Lanczos algorithm numerically, e.g., using time discretization approximations. 

As shown in \cite{Giscard2015}, there exists an isometry $\Phi$ between the algebra of distributions of $\D$ equipped with the $\ast$-product and the algebra of \emph{time-continuous} operators (for which the time variables $t'$ and $t$ serve as line and row indices). Consider, for simplicity, a discretization of the interval $I$ with constant time step $\Delta t$; then these operators become ordinary matrices.
Specifically, given $f, g \in \Sm$, their discretization counterparts are the \textit{lower triangular} matrices $\mathsf{F}, \mathsf{G}$. Moreover, the function $f \ast g$ corresponds to the matrix $\mathsf{F} \, \mathsf{G} \, \Delta t$, with the usual matrix product. 
In other terms, the isometry $\Phi$ followed by a time discretization sends the $\ast$-product to the ordinary matrix product times $\Delta t$. Similarly, the Dirac delta distribution is sent to the identity matrix times $1/(\Delta t)$, the $k$th Dirac delta derivative $\delta^{(k)}$ is sent to the finite difference matrix
$$
\big(\mathsf{M}_{\delta^{(k)}}\big)_{ij}=\frac{1}{(\Delta t)^{k+1}}\begin{cases}
(-1)^{i-j}\binom{k}{i-j},&\text{if }i\geq j\\
0,&\text{else}
\end{cases},
$$
and $\Theta$ is sent to the matrix
$
(\mathsf{M}_\Theta)_{ij}=1
$
if $i\geq j$ and 0 otherwise. 
Most importantly, in this picture, the $\ast$-inverse of a function $f(t',t)\in D(I)$ is given as $\mathsf{F}^{-1}/(\Delta t)^2$, with $\mathsf{F}$ the lower triangular matrix corresponding to $f$.
Moreover, the time-discretized version of the path-sum formulation Eq.~(\ref{PSresult}) only involves ordinary matrix resolvents. At the same time, the final integration of $\mathsf{R}_\ast(\mathsf{T}_n)_{11}$ yielding $\boldsymbol{w}^H\mathsf{U}\,\boldsymbol{v}$ becomes a left multiplication by $\mathsf{M}_\Theta$. Therefore, a numerical implementation of the time-discretized $\ast$-Lanczos algorithm only requires \emph{ordinary operations on triangular matrices}.

We can now meaningfully evaluate the numerical cost of the time-discretized version of the algorithm. Let $N_t$ be the number of time subintervals in the discretization of $I$ for both the $t$ and $t'$ time variables. Then time-discretized  $\ast$-multiplications or $\ast$-inversions cost $O(N_t^3)$ operations. 
Considering a sparse time-dependent matrix $\mathsf{A}(t)$ with $N_{nz}$ nonzero elements, the $\ast$-Lanczos algorithm therefore necessitates $O(N_i\times N_t^3\times N_{nz})$ operations to obtain the desired $\boldsymbol{w}^H\mathsf{U}\, \boldsymbol{v}$. 
Here $N_i$ is the number of iterations needed to get an error lower than a given tolerance. Unfortunately, as well-explained in \cite{LieStrBook13}, the presence of computational errors can slow down the (usual) Lanczos algorithm convergence. Hence, in general, we cannot assume $N_i \approx N$ since the $\ast$-Lanczos algorithm could analogously require more iterations. However, in many cases, the (usual) Lanczos algorithm demands \emph{few} iterations to reach the tolerance also in finite precision arithmetic. We expect the $\ast$-Lanczos algorithm to behave analogously, giving $N_i \ll N$ in many cases.  
Concerning $N_t$, there is no reason to expect that is would depend on $N$ since $N_t$ controls the quality  of individual generalized functions. 
We also remark that the $\ast$-Lanczos algorithm can exploit the sparsity structure of the matrix $\mathsf{A}$, making it inherently competitive when dealing with large sparse matrices that are typical of applications.

The classical numerical methods (e.g., Runge--Kutta methods) for the approximation of the system of ODEs \eqref{FundamentalSystem} are known to perform poorly in certain cases. These include for example, very large system sizes, or in the presence of highly oscillatory coefficients. Consequently, in the last decades, novel techniques have been sought and proposed, many of which are based on the Magnus series; see, for instance, \cite{HocLub99,BudAl99,IseAl2000,Ise02,Ise04,DegSch06,Cohetal2006,Blanes2009,BadEtAl16,Bla17}.
However, for large matrices, these methods are known to be highly consuming in resources.
This motivates the research of novel approaches in particular for large-scale problems.
Here the guaranteed convergence of the $\ast$-Lanczos algorithm in a finite number of iterations, the sequence of approximations it produces, its ability to exploit matrix sparsity and its relations with numerically well studied Lanczos procedures are all promising elements which justify further works on concrete numerical implementations.
 More precise theoretical results about the overall approximation quality and further issues on numerical applications of the present algorithm are beyond the scope of this work. They will appear together with a practical implementation of the algorithm in a future contribution. 
 
\section{Conclusion}\label{sec:conc}
In this work, we constructed the $\ast$-Lanczos algorithm as a biorthogonalization process of Krylov subspaces composed of distributions with respect to the $\ast$-product, a convolution-like product.
 The algorithm relies on a non-commutative operation and is analogous in spirit to the non-Hermitian Lanczos algorithm.
 The $\ast$-Lanczos algorithm can express the element of a time-ordered exponential of size $N \times N$ by the path-sum continued fraction \eqref{PSresult}. To our knowledge, such an expression is the only one composed of $\mathcal{O}(N)$ scalar integro-differential equations.
 Such a time-ordered exponential approximation relies on the matching
  moment property proved in this paper.
 
 The overall approach generates a controllable sequence of time-ordered exponential approximations, offers an innovative perspective of the connection between numerical linear algebra and differential calculus, and opens the door to efficient numerical algorithms for large-scale computations.

\backmatter

\bmhead{Acknowledgments}
We thank Francesca Arrigo, Des Higham, Jennifer Pestana, and Francesco Tudisco for their invitation to the  University of Strathclyde, without which this work would not have come to fruition. 
The first author was supported in part by 2019 Alcohol project ANR-19-CE40-0006 and 2020 Magica project ANR-20-CE29-0007. The second author was supported by Charles University Research programs No. PRIMUS/21/SCI/009 and UNCE/SCI/023.

\section*{Declarations}
\begin{itemize}
\item On behalf of all authors, the corresponding author states that there is no conflict of interest. 
%
\end{itemize}

\bibliography{Gauss_quadrature_time_exp}


\begin{thebibliography}{59}
\ifx \bisbn   \undefined \def \bisbn  #1{ISBN #1}\fi
\ifx \binits  \undefined \def \binits#1{#1}\fi
\ifx \bauthor  \undefined \def \bauthor#1{#1}\fi
\ifx \batitle  \undefined \def \batitle#1{#1}\fi
\ifx \bjtitle  \undefined \def \bjtitle#1{#1}\fi
\ifx \bvolume  \undefined \def \bvolume#1{\textbf{#1}}\fi
\ifx \byear  \undefined \def \byear#1{#1}\fi
\ifx \bissue  \undefined \def \bissue#1{#1}\fi
\ifx \bfpage  \undefined \def \bfpage#1{#1}\fi
\ifx \blpage  \undefined \def \blpage #1{#1}\fi
\ifx \burl  \undefined \def \burl#1{\textsf{#1}}\fi
\ifx \doiurl  \undefined \def \doiurl#1{\url{https://doi.org/#1}}\fi
\ifx \betal  \undefined \def \betal{\textit{et al.}}\fi
\ifx \binstitute  \undefined \def \binstitute#1{#1}\fi
\ifx \binstitutionaled  \undefined \def \binstitutionaled#1{#1}\fi
\ifx \bctitle  \undefined \def \bctitle#1{#1}\fi
\ifx \beditor  \undefined \def \beditor#1{#1}\fi
\ifx \bpublisher  \undefined \def \bpublisher#1{#1}\fi
\ifx \bbtitle  \undefined \def \bbtitle#1{#1}\fi
\ifx \bedition  \undefined \def \bedition#1{#1}\fi
\ifx \bseriesno  \undefined \def \bseriesno#1{#1}\fi
\ifx \blocation  \undefined \def \blocation#1{#1}\fi
\ifx \bsertitle  \undefined \def \bsertitle#1{#1}\fi
\ifx \bsnm \undefined \def \bsnm#1{#1}\fi
\ifx \bsuffix \undefined \def \bsuffix#1{#1}\fi
\ifx \bparticle \undefined \def \bparticle#1{#1}\fi
\ifx \barticle \undefined \def \barticle#1{#1}\fi
\bibcommenthead
\ifx \bconfdate \undefined \def \bconfdate #1{#1}\fi
\ifx \botherref \undefined \def \botherref #1{#1}\fi
\ifx \url \undefined \def \url#1{\textsf{#1}}\fi
\ifx \bchapter \undefined \def \bchapter#1{#1}\fi
\ifx \bbook \undefined \def \bbook#1{#1}\fi
\ifx \bcomment \undefined \def \bcomment#1{#1}\fi
\ifx \oauthor \undefined \def \oauthor#1{#1}\fi
\ifx \citeauthoryear \undefined \def \citeauthoryear#1{#1}\fi
\ifx \endbibitem  \undefined \def \endbibitem {}\fi
\ifx \bconflocation  \undefined \def \bconflocation#1{#1}\fi
\ifx \arxivurl  \undefined \def \arxivurl#1{\textsf{#1}}\fi
\csname PreBibitemsHook\endcsname

\bibitem{dyson1952}
\begin{barticle}
\bauthor{\bsnm{Dyson}, \binits{F.J.}}:
\batitle{Divergence of perturbation theory in quantum electrodynamics}.
\bjtitle{Phys. Rev.}
\bvolume{85}(\bissue{4}),
\bfpage{631}--\blpage{632}
(\byear{1952}).
\doiurl{10.1103/PhysRev.85.631}
\end{barticle}
\endbibitem

\bibitem{Xie2010}
\begin{barticle}
\bauthor{\bsnm{Xie}, \binits{Q.}},
\bauthor{\bsnm{Hai}, \binits{W.}}:
\batitle{Analytical results for a monochromatically driven two-level system}.
\bjtitle{Phys. Rev. A}
\bvolume{82},
\bfpage{032117}
(\byear{2010})
\end{barticle}
\endbibitem

\bibitem{Hortacsu2018}
\begin{barticle}
\bauthor{\bsnm{Horta\c{c}su}, \binits{M.}}:
\batitle{Heun functions and some of their applications in physics}.
\bjtitle{Adv. High Energy Phys.}
\bvolume{2018},
\bfpage{8621573}
(\byear{2018})
\end{barticle}
\endbibitem

\bibitem{Blanes2009}
\begin{barticle}
\bauthor{\bsnm{Blanes}, \binits{S.}},
\bauthor{\bsnm{Casas}, \binits{F.}},
\bauthor{\bsnm{Oteo}, \binits{J.A.}},
\bauthor{\bsnm{Ros}, \binits{J.}}:
\batitle{The {Magnus} expansion and some of its applications}.
\bjtitle{Phys. Rep.}
\bvolume{470}(\bissue{5}),
\bfpage{151}--\blpage{238}
(\byear{2009})
\end{barticle}
\endbibitem

\bibitem{Autler1955}
\begin{barticle}
\bauthor{\bsnm{Autler}, \binits{S.H.}},
\bauthor{\bsnm{Townes}, \binits{C.H.}}:
\batitle{Stark effect in rapidly varying fields}.
\bjtitle{Phys. Rev.}
\bvolume{100},
\bfpage{703}--\blpage{722}
(\byear{1955})
\end{barticle}
\endbibitem

\bibitem{Shirley1965}
\begin{barticle}
\bauthor{\bsnm{Shirley}, \binits{J.H.}}:
\batitle{Solution of the {Schr\"odinger} equation with a {Hamiltonian} periodic
  in time}.
\bjtitle{Phys. Rev.}
\bvolume{138},
\bfpage{979}--\blpage{987}
(\byear{1965})
\end{barticle}
\endbibitem

\bibitem{Lauder1986}
\begin{barticle}
\bauthor{\bsnm{Lauder}, \binits{M.A.}},
\bauthor{\bsnm{Knight}, \binits{P.L.}},
\bauthor{\bsnm{Greenland}, \binits{P.T.}}:
\batitle{Pulse-shape effects in intense-field laser excitation of atoms}.
\bjtitle{Opt. Acta}
\bvolume{33}(\bissue{10}),
\bfpage{1231}--\blpage{1252}
(\byear{1986})
\end{barticle}
\endbibitem

\bibitem{Reid63}
\begin{barticle}
\bauthor{\bsnm{Reid}, \binits{W.T.}}:
\batitle{Riccati matrix differential equations and non-oscillation criteria for
  associated linear differential systems}.
\bjtitle{Pacific J. Math.}
\bvolume{13}(\bissue{2}),
\bfpage{665}--\blpage{685}
(\byear{1963})
\end{barticle}
\endbibitem

\bibitem{kwaSiv72}
\begin{bbook}
\bauthor{\bsnm{Kwakernaak}, \binits{H.}},
\bauthor{\bsnm{Sivan}, \binits{R.}}:
\bbtitle{Linear Optimal Control Systems}
vol. \bseriesno{1}.
\bpublisher{Wiley-interscience},
\blocation{New York}
(\byear{1972})
\end{bbook}
\endbibitem

\bibitem{Corless2003}
\begin{bbook}
\bauthor{\bsnm{{Corless}}, \binits{M.}},
\bauthor{\bsnm{{Frazho}}, \binits{A.}}:
\bbtitle{{Linear Systems and Control: an Operator Perspective}}.
\bsertitle{{Pure and Applied Mathematics}}.
\bpublisher{{Marcel Dekker}},
\blocation{New York}
(\byear{2003})
\end{bbook}
\endbibitem

\bibitem{Blanes15}
\begin{barticle}
\bauthor{\bsnm{Blanes}, \binits{S.}}:
\batitle{{High order structure preserving explicit methods for solving
  linear-quadratic optimal control problems}}.
\bjtitle{Numer. Algorithms}
\bvolume{69}(\bissue{2}),
\bfpage{271}--\blpage{290}
(\byear{2015})
\end{barticle}
\endbibitem

\bibitem{BenEtAll17}
\begin{bbook}
\bauthor{\bsnm{Benner}, \binits{P.}},
\bauthor{\bsnm{Cohen}, \binits{A.}},
\bauthor{\bsnm{Ohlberger}, \binits{M.}},
\bauthor{\bsnm{Willcox}, \binits{K.}}:
\bbtitle{Model Reduction and Approximation: Theory and Algorithms}.
\bsertitle{Computational Science and Engineering}.
\bpublisher{SIAM},
\blocation{Philadelphia}
(\byear{2017})
\end{bbook}
\endbibitem

\bibitem{Kuvcera73}
\begin{barticle}
\bauthor{\bsnm{Ku{\v{c}}era}, \binits{V.}}:
\batitle{A review of the matrix {Riccati} equation}.
\bjtitle{Kybernetika}
\bvolume{9}(\bissue{1}),
\bfpage{42}--\blpage{61}
(\byear{1973})
\end{barticle}
\endbibitem

\bibitem{Abou2003}
\begin{bbook}
\bauthor{\bsnm{{Abou-Kandil}}, \binits{H.}},
\bauthor{\bsnm{Freiling}, \binits{G.}},
\bauthor{\bsnm{Ionescu}, \binits{V.}},
\bauthor{\bsnm{Jank}, \binits{G.}}:
\bbtitle{Matrix {{Riccati}} Equations in Control and Systems Theory}.
\bsertitle{Systems \& {{Control}}: {{Foundations}} \& {{Applications}}}.
\bpublisher{{Birkh{\"a}user}},
\blocation{Basel}
(\byear{2003})
\end{bbook}
\endbibitem

\bibitem{hached2018}
\begin{barticle}
\bauthor{\bsnm{Hached}, \binits{M.}},
\bauthor{\bsnm{Jbilou}, \binits{K.}}:
\batitle{Numerical solutions to large-scale differential {{Lyapunov}} matrix
  equations}.
\bjtitle{Numer. Algorithms}
\bvolume{79}(\bissue{3}),
\bfpage{741}--\blpage{757}
(\byear{2018})
\end{barticle}
\endbibitem

\bibitem{KirSim19}
\begin{barticle}
\bauthor{\bsnm{Kirsten}, \binits{G.}},
\bauthor{\bsnm{Simoncini}, \binits{V.}}:
\batitle{Order reduction methods for solving large-scale differential matrix
  {R}iccati equations}.
\bjtitle{SIAM J. Sci. Comput.}
\bvolume{42}(\bissue{4}),
\bfpage{2182}--\blpage{2205}
(\byear{2020}).
\doiurl{10.1137/19m1264217}
\end{barticle}
\endbibitem

\bibitem{GisPoz21}
\begin{barticle}
\bauthor{\bsnm{Giscard}, \binits{P.-L.}},
\bauthor{\bsnm{Pozza}, \binits{S.}}:
\batitle{Tridiagonalization of systems of coupled linear differential equations
  with variable coefficients by a {L}anczos-like method}.
\bjtitle{Linear Algebra and its Applications}
\bvolume{624},
\bfpage{153}--\blpage{173}
(\byear{2021}).
\doiurl{10.1016/j.laa.2021.04.011}
\end{barticle}
\endbibitem

\bibitem{GisPozInv19}
\begin{barticle}
\bauthor{\bsnm{Giscard}, \binits{P.-L.}},
\bauthor{\bsnm{Pozza}, \binits{S.}}:
\batitle{Lanczos-like algorithm for the time-ordered exponential: The
  $\ast$-inverse problem}.
\bjtitle{Appl. Math.}
\bvolume{65}(\bissue{6}),
\bfpage{807}--\blpage{827}
(\byear{2020}).
\doiurl{10.21136/am.2020.0342-19}
\end{barticle}
\endbibitem

\bibitem{Magnus1954}
\begin{barticle}
\bauthor{\bsnm{Magnus}, \binits{W.}}:
\batitle{On the exponential solution of differential equations for a linear
  operator}.
\bjtitle{Comm. Pure Appl. Math.}
\bvolume{7}(\bissue{4}),
\bfpage{649}--\blpage{673}
(\byear{1954})
\end{barticle}
\endbibitem

\bibitem{Casas07}
\begin{barticle}
\bauthor{\bsnm{Casas}, \binits{F.}}:
\batitle{Sufficient conditions for the convergence of the {M}agnus expansion}.
\bjtitle{J. Phys. A}
\bvolume{40}(\bissue{50}),
\bfpage{15001}--\blpage{15017}
(\byear{2007}).
\doiurl{10.1088/1751-8113/40/50/006}
\end{barticle}
\endbibitem

\bibitem{Feldman1984}
\begin{barticle}
\bauthor{\bsnm{Fel'dman}, \binits{E.B.}}:
\batitle{On the convergence of the {M}agnus expansion for spin systems in
  periodic magnetic fields}.
\bjtitle{Phys. Lett.}
\bvolume{104A}(\bissue{9}),
\bfpage{479}--\blpage{481}
(\byear{1984})
\end{barticle}
\endbibitem

\bibitem{Maricq1987}
\begin{barticle}
\bauthor{\bsnm{Maricq}, \binits{M.M.}}:
\batitle{Convergence of the {Magnus} expansion for time dependent two level
  systems}.
\bjtitle{J. Chem. Phys.}
\bvolume{86}(\bissue{10}),
\bfpage{5647}--\blpage{5651}
(\byear{1987})
\end{barticle}
\endbibitem

\bibitem{IseAl2000}
\begin{barticle}
\bauthor{\bsnm{Iserles}, \binits{A.}},
\bauthor{\bsnm{Munthe-Kaas}, \binits{H.Z.}},
\bauthor{\bsnm{Nørsett}, \binits{S.P.}},
\bauthor{\bsnm{Zanna}, \binits{A.}}:
\batitle{Lie-group methods}.
\bjtitle{Acta Numer.}
\bvolume{9},
\bfpage{215}--\blpage{365}
(\byear{2000}).
\doiurl{10.1017/S0962492900002154}
\end{barticle}
\endbibitem

\bibitem{MoaNie07}
\begin{barticle}
\bauthor{\bsnm{Moan}, \binits{P.C.}},
\bauthor{\bsnm{Niesen}, \binits{J.}}:
\batitle{Convergence of the {M}agnus series}.
\bjtitle{Found. Comput. Math.}
\bvolume{8}(\bissue{3}),
\bfpage{291}--\blpage{301}
(\byear{2007}).
\doiurl{10.1007/s10208-007-9010-0}
\end{barticle}
\endbibitem

\bibitem{Sanchez2011}
\begin{barticle}
\bauthor{\bsnm{S{\'a}nchez}, \binits{S.}},
\bauthor{\bsnm{Casas}, \binits{F.}},
\bauthor{\bsnm{Fern{\'a}ndez}, \binits{A.}}:
\batitle{New analytic approximations based on the {Magnus} expansion}.
\bjtitle{J. Math. Chem.}
\bvolume{49}(\bissue{8}),
\bfpage{1741}--\blpage{1758}
(\byear{2011})
\end{barticle}
\endbibitem

\bibitem{Giscard2015}
\begin{barticle}
\bauthor{\bsnm{Giscard}, \binits{P.-L.}},
\bauthor{\bsnm{Lui}, \binits{K.}},
\bauthor{\bsnm{Thwaite}, \binits{S.J.}},
\bauthor{\bsnm{Jaksch}, \binits{D.}}:
\batitle{An exact formulation of the time-ordered exponential using path-sums}.
\bjtitle{J. Math. Phys.}
\bvolume{56}(\bissue{5}),
\bfpage{053503}
(\byear{2015})
\end{barticle}
\endbibitem

\bibitem{Balasubramanian2020}
\begin{botherref}
\oauthor{\bsnm{Balasubramanian}, \binits{V.}},
\oauthor{\bsnm{DeCross}, \binits{M.}},
\oauthor{\bsnm{Kar}, \binits{A.}},
\oauthor{\bsnm{Parrikar}, \binits{O.}}:
Quantum complexity of time evolution with chaotic {H}amiltonians.
J. High Energ. Phys.
\textbf{134}
(2020)
\end{botherref}
\endbibitem

\bibitem{BonGis2020}
\begin{barticle}
\bauthor{\bsnm{Giscard}, \binits{P.-L.}},
\bauthor{\bsnm{Bonhomme}, \binits{C.}}:
\batitle{Dynamics of quantum systems driven by time-varying {H}amiltonians:
  Solution for the {B}loch-{S}iegert {H}amiltonian and applications to {NMR}}.
\bjtitle{Phys. Rev. Research}
\bvolume{2},
\bfpage{023081}
(\byear{2020}).
\doiurl{10.1103/PhysRevResearch.2.023081}
\end{barticle}
\endbibitem

\bibitem{Flum2004}
\begin{barticle}
\bauthor{\bsnm{J.~Flum}, \binits{M.G.}}:
\batitle{The parameterized complexity of counting problems}.
\bjtitle{SIAM J. Comput.}
\bvolume{33},
\bfpage{892}--\blpage{922}
(\byear{2004})
\end{barticle}
\endbibitem

\bibitem{MolVLo78}
\begin{barticle}
\bauthor{\bsnm{Moler}, \binits{C.}},
\bauthor{\bsnm{{Van Loan}}, \binits{C.}}:
\batitle{Nineteen dubious ways to compute the exponential of a matrix}.
\bjtitle{SIAM Rev.}
\bvolume{20}(\bissue{4}),
\bfpage{801}--\blpage{836}
(\byear{1978})
\end{barticle}
\endbibitem

\bibitem{MolVLo03}
\begin{barticle}
\bauthor{\bsnm{Moler}, \binits{C.}},
\bauthor{\bsnm{{Van Loan}}, \binits{C.}}:
\batitle{Nineteen dubious ways to compute the exponential of a matrix,
  twenty-five years later}.
\bjtitle{SIAM Rev.}
\bvolume{45}(\bissue{1}),
\bfpage{3}--\blpage{49}
(\byear{2003})
\end{barticle}
\endbibitem

\bibitem{HigBook08}
\begin{bbook}
\bauthor{\bsnm{Higham}, \binits{N.J.}}:
\bbtitle{Functions of Matrices. {T}heory and Computation},
p. \bfpage{425}.
\bpublisher{SIAM},
\blocation{Philadelphia}
(\byear{2008})
\end{bbook}
\endbibitem

\bibitem{Gut92}
\begin{barticle}
\bauthor{\bsnm{Gutknecht}, \binits{M.H.}}:
\batitle{{A completed theory of the unsymmetric {L}anczos process and related
  algorithms. {I}}}.
\bjtitle{SIAM J. Matrix Anal. Appl.}
\bvolume{13}(\bissue{2}),
\bfpage{594}--\blpage{639}
(\byear{1992})
\end{barticle}
\endbibitem

\bibitem{Gut94b}
\begin{barticle}
\bauthor{\bsnm{Gutknecht}, \binits{M.H.}}:
\batitle{{A completed theory of the unsymmetric {L}anczos process and related
  algorithms. {II}}}.
\bjtitle{SIAM J. Matrix Anal. Appl.}
\bvolume{15}(\bissue{1}),
\bfpage{15}--\blpage{58}
(\byear{1994})
\end{barticle}
\endbibitem

\bibitem{GolMeuBook10}
\begin{bbook}
\bauthor{\bsnm{Golub}, \binits{G.H.}},
\bauthor{\bsnm{Meurant}, \binits{G.}}:
\bbtitle{Matrices, Moments and Quadrature with Applications}.
\bsertitle{{Princeton Ser. Appl. Math.}},
p. \bfpage{363}.
\bpublisher{Princeton University Press},
\blocation{Princeton}
(\byear{2010})
\end{bbook}
\endbibitem

\bibitem{LieStrBook13}
\begin{bbook}
\bauthor{\bsnm{Liesen}, \binits{J.}},
\bauthor{\bsnm{Strakoš}, \binits{Z.}}:
\bbtitle{{Krylov} Subspace Methods: Principles and Analysis}.
\bsertitle{{Numer. Math. Sci. Comput.}}
\bpublisher{Oxford University Press},
\blocation{Oxford}
(\byear{2013})
\end{bbook}
\endbibitem

\bibitem{PozPraStr16}
\begin{barticle}
\bauthor{\bsnm{Pozza}, \binits{S.}},
\bauthor{\bsnm{Pranić}, \binits{M.S.}},
\bauthor{\bsnm{Strakoš}, \binits{Z.}}:
\batitle{{Gauss quadrature for quasi-definite linear functionals}}.
\bjtitle{IMA J. Numer. Anal.}
\bvolume{37}(\bissue{3}),
\bfpage{1468}--\blpage{1495}
(\byear{2017})
\end{barticle}
\endbibitem

\bibitem{PozPraStr18}
\begin{barticle}
\bauthor{\bsnm{Pozza}, \binits{S.}},
\bauthor{\bsnm{Pranić}, \binits{M.S.}},
\bauthor{\bsnm{Strakoš}, \binits{Z.}}:
\batitle{The {L}anczos algorithm and complex {G}auss quadrature}.
\bjtitle{Electron. Trans. Numer. Anal.}
\bvolume{50},
\bfpage{1}--\blpage{19}
(\byear{2018})
\end{barticle}
\endbibitem

\bibitem{Par92}
\begin{barticle}
\bauthor{\bsnm{Parlett}, \binits{B.N.}}:
\batitle{{Reduction to tridiagonal form and minimal realizations}}.
\bjtitle{SIAM J. Matrix Anal. Appl.}
\bvolume{13}(\bissue{2}),
\bfpage{567}--\blpage{593}
(\byear{1992})
\end{barticle}
\endbibitem

\bibitem{DraBook83}
\begin{bbook}
\bauthor{\bsnm{Draux}, \binits{A.}}:
\bbtitle{{Polynômes Orthogonaux Formels}}.
\bsertitle{{Lecture Notes in Math.}},
vol. \bseriesno{974},
p. \bfpage{625}.
\bpublisher{Springer},
\blocation{Berlin}
(\byear{1983})
\end{bbook}
\endbibitem

\bibitem{schwartz1952}
\begin{bbook}
\bauthor{\bsnm{Halperin}, \binits{I.}},
\bauthor{\bsnm{Schwartz}, \binits{L.}}:
\bbtitle{Introduction to the Theory of Distributions}.
\bpublisher{University of Toronto Press},
\blocation{Toronto}
(\byear{19 Feb. 2019}).
\doiurl{10.3138/9781442615151}.
\burl{https://toronto.degruyter.com/view/title/550976}
\end{bbook}
\endbibitem

\bibitem{schwartz1978}
\begin{bbook}
\bauthor{\bsnm{Schwartz}, \binits{L.}}:
\bbtitle{Th{\'e}orie Des Distributions},
\bedition{Nouvelle {\'e}dition, enti{\`e}rement corrig{\'e}e, refondue et
  augment{\'e}e} edn.
\bpublisher{{Hermann}},
\blocation{{Paris}}
(\byear{1978})
\end{bbook}
\endbibitem

\bibitem{Volterra1928}
\begin{bbook}
\bauthor{\bsnm{{Volterra}}, \binits{V.}},
\bauthor{\bsnm{{Pérès}}, \binits{J.}}:
\bbtitle{{Leçons sur la Composition et les Fonctions Permutables}}.
\bpublisher{{Éditions Jacques Gabay}},
\blocation{Paris}
(\byear{1928})
\end{bbook}
\endbibitem

\bibitem{Giscard2012}
\begin{botherref}
\oauthor{\bsnm{Giscard}, \binits{P.-L.}},
\oauthor{\bsnm{Thwaite}, \binits{S.J.}},
\oauthor{\bsnm{Jaksch}, \binits{D.}}:
Walk-sums, continued fractions and unique factorisation on digraphs.
arXiv:1202.5523 [cs.DM]
(2012)
\end{botherref}
\endbibitem

\bibitem{Kilic2008}
\begin{barticle}
\bauthor{\bsnm{K{\i}l{\i}\c{c}}, \binits{E.}}:
\batitle{Explicit formula for the inverse of a tridiagonal matrix by backward
  continued fractions}.
\bjtitle{Appl. Math. Comput.}
\bvolume{197}(\bissue{1}),
\bfpage{345}--\blpage{357}
(\byear{2008})
\end{barticle}
\endbibitem

\bibitem{Tay82}
\begin{botherref}
\oauthor{\bsnm{Taylor}, \binits{D.R.}}:
{Analysis of the look ahead {L}anczos algorithm}.
PhD thesis,
University of California,
Berkeley
(1982)
\end{botherref}
\endbibitem

\bibitem{ParTayLiu85}
\begin{barticle}
\bauthor{\bsnm{Parlett}, \binits{B.N.}},
\bauthor{\bsnm{Taylor}, \binits{D.R.}},
\bauthor{\bsnm{Liu}, \binits{Z.A.}}:
\batitle{{A look-ahead {L}anczos algorithm for unsymmetric matrices}}.
\bjtitle{Math. Comp.}
\bvolume{44}(\bissue{169}),
\bfpage{105}--\blpage{124}
(\byear{1985})
\end{barticle}
\endbibitem

\bibitem{Freund1993}
\begin{barticle}
\bauthor{\bsnm{Freund}, \binits{R.W.}},
\bauthor{\bsnm{Gutknecht}, \binits{M.H.}},
\bauthor{\bsnm{Nachtigal}, \binits{N.M.}}:
\batitle{An implementation of the look-ahead {Lanczos} algorithm for
  non-{Hermitian} matrices}.
\bjtitle{SIAM J. Sci. Comput.}
\bvolume{14},
\bfpage{137}--\blpage{158}
(\byear{1993})
\end{barticle}
\endbibitem

\bibitem{BreRedSad91}
\begin{barticle}
\bauthor{\bsnm{Brezinski}, \binits{C.}},
\bauthor{\bsnm{{Redivo Zaglia}}, \binits{M.}},
\bauthor{\bsnm{Sadok}, \binits{H.}}:
\batitle{{Avoiding breakdown and near-breakdown in {L}anczos type algorithms}}.
\bjtitle{Numer. Algorithms}
\bvolume{1}(\bissue{3}),
\bfpage{261}--\blpage{284}
(\byear{1991})
\end{barticle}
\endbibitem

\bibitem{Brezinski1992}
\begin{barticle}
\bauthor{\bsnm{Brezinski}, \binits{C.}},
\bauthor{\bsnm{{Redivo Zaglia}}, \binits{M.}},
\bauthor{\bsnm{Sadok}, \binits{H.}}:
\batitle{{A breakdown-free Lanczos type algorithm for solving linear systems}}.
\bjtitle{Numer. Math.}
\bvolume{63}(\bissue{1}),
\bfpage{29}--\blpage{38}
(\byear{1992})
\end{barticle}
\endbibitem

\bibitem{PozPra19}
\begin{barticle}
\bauthor{\bsnm{Pozza}, \binits{S.}},
\bauthor{\bsnm{Prani{\'{c}}}, \binits{M.}}:
\batitle{The {G}auss quadrature for general linear functionals, {L}anczos
  algorithm, and minimal partial realization}.
\bjtitle{Numer. Algor}
\bvolume{88},
\bfpage{647}--\blpage{678}
(\byear{2021}).
\doiurl{10.1007/s11075-020-01052-y}
\end{barticle}
\endbibitem

\bibitem{HocLub99}
\begin{barticle}
\bauthor{\bsnm{Hochbruck}, \binits{M.}},
\bauthor{\bsnm{Lubich}, \binits{C.}}:
\batitle{Exponential integrators for quantum-classical molecular dynamics}.
\bjtitle{BIT Numer. Math.}
\bvolume{39}(\bissue{4}),
\bfpage{620}--\blpage{645}
(\byear{1999}).
\doiurl{10.1023/A:1022335122807}
\end{barticle}
\endbibitem

\bibitem{BudAl99}
\begin{barticle}
\bauthor{\bsnm{Budd}, \binits{C.J.}},
\bauthor{\bsnm{Iserles}, \binits{A.}},
\bauthor{\bsnm{Iserles}, \binits{A.}},
\bauthor{\bsnm{Nørsett}, \binits{S.P.}}:
\batitle{On the solution of linear differential equations in lie groups}.
\bjtitle{Philosophical Transactions of the Royal Society of London. Series A:
  Mathematical, Physical and Engineering Sciences}
\bvolume{357}(\bissue{1754}),
\bfpage{983}--\blpage{1019}
(\byear{1999})
{\href{https://arxiv.org/abs/https://royalsocietypublishing.org/doi/pdf/10.1098/rsta.1999.0362}{{https://royalsocietypublishing.org/doi/pdf/10.1098/rsta.1999.0362}}}.
\doiurl{10.1098/rsta.1999.0362}
\end{barticle}
\endbibitem

\bibitem{Ise02}
\begin{barticle}
\bauthor{\bsnm{Iserles}, \binits{A.}}:
\batitle{On the global error of discretization methods for highly-oscillatory
  ordinary differential equations}.
\bjtitle{BIT Numer. Math.}
\bvolume{42}(\bissue{3}),
\bfpage{561}--\blpage{599}
(\byear{2002}).
\doiurl{10.1023/A:1022049814688}
\end{barticle}
\endbibitem

\bibitem{Ise04}
\begin{barticle}
\bauthor{\bsnm{Iserles}, \binits{A.}}:
\batitle{On the method of {N}eumann series for highly oscillatory equations}.
\bjtitle{BIT Numer. Math.}
\bvolume{44}(\bissue{3}),
\bfpage{473}--\blpage{488}
(\byear{2004}).
\doiurl{10.1023/B:BITN.0000046810.25353.95}
\end{barticle}
\endbibitem

\bibitem{DegSch06}
\begin{barticle}
\bauthor{\bsnm{Degani}, \binits{I.}},
\bauthor{\bsnm{Schiff}, \binits{J.}}:
\batitle{{RCMS}: Right correction {M}agnus series approach for oscillatory
  {ODE}s}.
\bjtitle{J. Comput. Appl. Math.}
\bvolume{193}(\bissue{2}),
\bfpage{413}--\blpage{436}
(\byear{2006}).
\doiurl{10.1016/j.cam.2005.07.001}
\end{barticle}
\endbibitem

\bibitem{Cohetal2006}
\begin{bchapter}
\bauthor{\bsnm{Cohen}, \binits{D.}},
\bauthor{\bsnm{Jahnke}, \binits{T.}},
\bauthor{\bsnm{Lorenz}, \binits{K.}},
\bauthor{\bsnm{Lubich}, \binits{C.}}:
\bctitle{Numerical integrators for highly oscillatory {H}amiltonian systems: A
  review}.
In: \beditor{\bsnm{Mielke}, \binits{A.}} (ed.)
\bbtitle{Analysis, Modeling and Simulation of Multiscale Problems},
pp. \bfpage{553}--\blpage{576}.
\bpublisher{Springer},
\blocation{Berlin, Heidelberg}
(\byear{2006})
\end{bchapter}
\endbibitem

\bibitem{BadEtAl16}
\begin{barticle}
\bauthor{\bsnm{Bader}, \binits{P.}},
\bauthor{\bsnm{Iserles}, \binits{A.}},
\bauthor{\bsnm{Kropielnicka}, \binits{K.}},
\bauthor{\bsnm{Singh}, \binits{P.}}:
\batitle{Efficient methods for linear {S}chr{\"o}dinger equation in the
  semiclassical regime with time-dependent potential}.
\bjtitle{Proceedings of the Royal Society A: Mathematical, Physical and
  Engineering Sciences}
\bvolume{472},
\bfpage{20150733}
(\byear{2016})
\end{barticle}
\endbibitem

\bibitem{Bla17}
\begin{bbook}
\bauthor{\bsnm{Blanes}, \binits{S.}},
\bauthor{\bsnm{Casas}, \binits{F.}}:
\bbtitle{A Concise Introduction to Geometric Numerical Integration}.
\bpublisher{CRC Press},
\blocation{Bocan Raton, FL}
(\byear{2017})
\end{bbook}
\endbibitem

\end{thebibliography}

\end{document}